\documentclass[reqno,12pt,letterpaper]{amsart}

\usepackage{amssymb,amsmath,amsthm,amsxtra,esint,mathtools,MnSymbol,mathrsfs,url}

\usepackage{graphicx}
\usepackage[usenames,dvipsnames]{color}

\usepackage[colorlinks=true,linkcolor=Red,citecolor=Green]{hyperref}

\mathtoolsset{showonlyrefs,showmanualtags} 

\usepackage{tablefootnote}

\usepackage{float}

\usepackage[font=small,labelfont=bf]{caption}
\usepackage[style=alphabetic, backend=biber,doi=false,isbn=false,url=false,eprint=false,giveninits=true,maxbibnames=99]{biblatex}
\def\?[#1]{\textbf{[#1]}\marginpar{\Large{\textbf{??}}}}
\def\smallsection#1{\smallskip\noindent\textbf{#1}.}

\newtheorem*{theorem*}{Theorem}
\newtheorem{theorem}{Theorem}
\newtheorem{prop}{Proposition}[section]
\newtheorem{defi}[prop]{Definition}
\newtheorem{hyp}{Hypothesis}

\newtheorem{lemma}[prop]{Lemma}

\newtheorem{qu}{Question}[section]

\newtheorem{claim}{Claim}[section]
\numberwithin{equation}{section}

\newcommand{\CC}{{\mathbb C}}

\newcommand{\ZZ}{{\mathbb Z}}

\newcommand*{\dd}{\mathop{}\!\mathrm{d}}

\DeclareMathOperator{\Spec}{Spec}

\DeclareMathOperator{\Op}{Op}
\DeclareMathOperator{\diag}{diag}
\DeclareMathOperator{\Vol}{Vol}
\DeclareMathOperator{\Span}{span}

\newcommand{\ip}[2]{\left  \langle#1,#2 \right \rangle}
\newcommand{\mat}[1]{\begin{pmatrix} #1 \end{pmatrix}}

\newcommand{\set}[1]{ \left \{ #1 \right  \}}
\newcommand{\N}{\mathbb{N}}
\newcommand{\p}{\partial}
\newcommand{\dbar}{\overline{\p}}
\newcommand{\Z}{\mathbb{Z}}
\newcommand{\T}{\mathbb{T}}

\newcommand{\e}{\varepsilon}

\newcommand{\R}{\mathbb{R}}

\newcommand{\norm}[1]{ \left \| #1 \right  \|  }
\newcommand{\jpb}[1]{ \left \langle #1  \right \rangle    }
\newcommand{\C}{\mathbb{C}}
\renewcommand{\phi}{\varphi}
\renewcommand{\P}{\mathbb{P}}

\renewcommand{\Re}[1]{{\rm{Re}} \left ( #1\right ) }
\renewcommand{\Im}[1]{{\rm{Im}} \left ( #1 \right ) }

\let\Im=\Imag

\let\Re=\Real

\newcommand{\abs}[1]{\left | #1 \right| }
\newcommand{\SAA}{\Theta} 
\newcommand{\SAAA}{\Theta} 
\newcommand{\edit}[1]{{#1}}

\title{Spectral Instability of Random Fredholm Operators}

\author{Simon Becker} 
\address[Simon Becker]{ETH Zurich, 
Institute for Mathematical Research, 
Rämistrasse 101, 8092 Zurich, 
Switzerland}
\email{simon.becker@math.ethz.ch}
\author{Izak Oltman} 
\address[Izak Oltman]{Department of Mathematics, Northwestern University, 2033 Sheridan Rd, Evanston, IL 60208}
\email{ioltman@northwestern.edu}
\author{Martin Vogel} 
\address[Martin Vogel]{Institut de Recherche Math{\'e}matique Avanc{\'e}e - UMR 7501, 
Universit{\'e} de Strasbourg et CNRS, 7 rue René-Descartes, 67084 Strasbourg Cedex, France.}
\email{vogel@math.unistra.fr}
\date{\today}

\begin{document}

\begin{abstract}
If $A \colon D(A) \subset \mathcal{H} \to \mathcal{H}$ is an {unbounded} Fredholm operator of index $0$ on a Hilbert space $\mathcal{H}$ with a dense domain $D(A)$, then its spectrum is either discrete or the entire complex plane. 
{This spectral dichotomy plays a central role in the study of \textit{magic angles} in twisted bilayer graphene.}

This paper proves that if {such} operators (with certain additional assumptions) are perturbed by certain random trace-class operators, their spectrum is discrete with high probability.
\end{abstract}

\maketitle

\section{Introduction}

This article describes the spectrum of random perturbations of Fredholm operators of index $0$.
If $A \colon D(A) \subset \mathcal{H} \to \mathcal{H}$ is an {unbounded}  Fredholm operator of index $0$ on a {separable} Hilbert space $\mathcal{H}$ with dense domain $D(A)$, then $\Spec(A)$ is either discrete or the entire complex plane (see Proposition \ref{prop1}).
\edit{
Here we denote $\Spec(A)$ as the spectrum of a linear operator on a Hilbert space, defined as the complement of the set of points $z\in \C$ such that $A-z$ is bijective.
}

{A striking example of this dichotomy was presented in \cite{seeley1986}, where Seeley introduced {a particularly simple} family of operators {with this property}}, defined as
\begin{align}
   A_{\text{Seeley}} \colon H^1(\R /2\pi \Z) \ni f \mapsto  a(x) \p_xf(x) + b(x) f(x) \in L^2 (\R / 2\pi \Z)\label{eq:seeleyoperator}
\end{align}
for $a,b \in C^\infty (\R /2\pi \Z)$ and $|a| > 0$. 
The spectrum of $A_{\text{Seeley}}$ is either a discrete lattice, empty, or $\C$ depending on $a$ and $b$.

{While the case where $\Spec(A) = \C$ was considered pathological, the recent study of physical models in \textit{twisted bilayer graphene} (TBG) have shown this case to be physically highly relevant.}
Indeed, this spectral dichotomy {appears} in the mathematical study of TBG in the so-called chiral limit, see \cite{tarnopolsky2019origin,Becker2020} for details (which is the main motivation for this article).
In this model, a Fredholm operator of index $0$ is defined with a parameter $\beta \in \C$.
The theory of TBG in the chiral limit defines \textit{magic} twisting angles as the $\alpha$ such that the spectrum of the operator is $\C$ (this is further described in \S \ref{sec:motiv}). 

For any Fredholm operator of index $0$, it is easy to construct an arbitrarily small perturbation to make the spectrum of such an operator discrete (by mapping the kernel to the cokernel). 
The question that this paper aims to address is 

\textit{{When is the spectrum} of a random perturbation of Fredholm operators of index zero on a Hilbert space {discrete?}}

The perturbations considered in this article are random trace-class perturbations.
These types of random perturbations appear {when} numerically {analyzing such operators}.
For example, if we restrict an operator $A$ to an $N$-dimensional vector space\footnote{For example, we could approximate $H^1(\R /2\pi \Z)$ by the first $N$ Fourier modes.} and project the image onto this same vector space, then we can easily compute the spectrum numerically.
As we increase $N$, rounding errors in mathematical software {can be modeled by} finite-rank random perturbations (see also \cite{references}).

In this paper (under certain conditions on the random perturbation), we provide a quantitative lower bound for the smallest singular value of the randomly perturbed operator (with certain probability). 
As a by-product, we see that these randomly perturbed operators have a discrete spectrum with high probability.
We then apply our findings to the TBG model (see Theorem \ref{Thm.1}).
A heuristic version of our main theorem is stated below; for a more precise version, see Theorem \ref{thm:general result}.
\begin{theorem*}[Heuristic Main Result]
Suppose $A - z$ is a Fredholm operator of index $0$ for all $z\in \Omega \subset \C$ (an open set) and $Q_\omega$ is a suitable random trace class perturbation (see \eqref{random_pertubation}), then for sufficiently small $\delta > 0$, the set $\Spec(A + \delta Q_\omega) \cap \Omega$ is discrete with high probability.
\end{theorem*}

The general framework of our proof follows the works of Hager, Sj\"ostrand, and the authors \cite{Hager2006,Hager2008,Vogel2020,oltman2023}. 
However, in all of those cases, the proofs rely on the semiclassical ellipticity of the symbol of the operator being studied.
{In this paper, we will work under more general assumptions of the operator to randomly perturb.
In the case the operator is a pseudo-differential operator, we allow the principal symbol to be \textit{not} semiclassicaly elliptic.}

The spectral stability of our main result can be characterized more abstractly in the following way.
{Recall {that} $\lambda \in \C$ is a normal eigenvalue of $A$ if $\lambda \in \Spec(A)$ is isolated and the kernel of $A-\lambda$ is finite dimensional.}
We then define the discrete spectrum of $A$ as
\begin{align}
    \Spec_{\text{disc}}(A) \coloneq \set{\lambda \in \C : \lambda \text{ is a normal eigenvalue of $A$}}
\end{align}

and its essential spectrum
\begin{align}
    \Spec_{\text{ess}} (A) \coloneq \Spec(A) \setminus \Spec_{\text{disc}}(A).
\end{align}
If $A$ is a self-adjoint operator, and $S$ is a relatively $A$-compact operator, then it is a classical theorem due to Weyl \cite{Weyl1910} that
\begin{align}
    \Spec_{\text{ess}}(A) = \Spec_{\text{ess}}(A +S ).
\end{align}

It is well known that this {does not hold in general}  for non-self-adjoint operators $A$. 
In fact for non-self-adjoint operators $A$, there exists a maximal set \cite{GUSTAFSON1969121} $\operatorname{Spec}_{\text{ess,stab}}(A)$ contained in $\operatorname{Spec}_{\text{ess}}(A)$ that is invariant under compact perturbations
\[ \operatorname{Spec}_{\text{ess,stab}}(A) = \bigcap_{S \text{ compact on }\mathcal H} \Spec(A+S).\]
One can verify that
\begin{equation}
\label{eq:characterization}
\operatorname{Spec}_{\text{ess,stab}}(A)=\{ \lambda: A-\lambda \text{ is not a Fredholm operator of index }0\}.
\end{equation}

We recall that a Fredholm operator on a Hilbert space $\mathcal H$ is a closed linear operator $A\colon D(A) \subset \mathcal H \to \mathcal H$ such that $\operatorname{ran}(A)$ is closed and both its kernel and its cokernel are finite-dimensional. Here, $D(A)$ is a dense linear subspace of {$\mathcal H$} that is complete with respect to the norm $\Vert x \Vert_{D(A)}:=\sqrt{\Vert x  \Vert^2 +\Vert Ax \Vert^2}.$
 
In the case of self-adjoint operators $A$, Weyl's theorem gives
\[\operatorname{Spec}_{\text{ess,stab}}(A)=\operatorname{Spec}_{\text{ess}}(A).\]

Equation \eqref{eq:characterization} raises the question about the stability of $ \operatorname{Spec}_{\text{ess}}(A) \setminus \operatorname{Spec}_{\text{ess,stab}}(A)$ under {generic} perturbations. 
By definition, this part of the spectrum is not stable under all compact perturbations, but is it stable under suitable random perturbations of $A$?

This article considers random perturbations of Fredholm operators of index $0$.
As we shall see in the following, the spectrum of such operators satisfies a dichotomy on the set of $z$ such that $A-z$ is Fredholm of index $0$. 
{We aim to provide quantitative estimates on the stability of the essential spectrum by estimating the probability the smallest singular value of $A -z$ is not too small.}

\begin{qu}
\label{qu:stab}
How stable is the essential spectrum of non-self-adjoint Fredholm operators of index $0$ under random trace-class perturbations?
\end{qu}

More specifically, for a fixed {separable} Hilbert space $\mathcal{H}$ with a dense subset $D(A)$, we consider unbounded Fredholm operators $A \colon D(A)\subset \mathcal H \to \mathcal{H}$ with index $0$. For such operators, let
\begin{align}
    \rho_F(A)\coloneqq \{z\in \CC:  A-z \text{ is a Fredholm operator}\} \label{eq:rhoFA}
\end{align}
denote the Fredholm domain of $A$. Because the Fredholm index is constant with respect to small perturbations (see for instance \cite[Theorem C.5]{dyatlov2019mathematical}), the Fredholm index is constant $\operatorname{ind}(A-z')=\operatorname{ind}(A-z)$ for $z,z'$ in the same connected component of $\rho_F(A)$ \edit{and $\rho_F(A)$ is an open set}.   We then define \edit{the open set}
\begin{align}
    \rho^{(0)}_F(A)\coloneqq \{z\in \CC:  A-z \text{ is a Fredholm operator of index } 0\} \label{eq:rhoFA0}.
\end{align}
Note that if $A$ is Fredholm of index zero, then trivially $0\in \rho^{(0)}_F(A)$.

Analytic Fredholm theory, see Proposition \ref{prop1}, shows that for each connected component $D \subset \rho_F^{(0)}(A)$, either $D\subset \Spec(A)$ or $\Spec(A) \cap D$ is discrete.
adWe can always add an arbitrarily small \edit{(ad hoc)} finite rank perturbation to $A$ so that the second case holds\footnote{\edit{Indeed, without loss of generality, suppose that $0\in D$. Then if the kernel of $A$ has an orthonormal basis $\set{u_i}_{i=1}^{N}$, and the cokernel of $A$ has an orthonormal basis $\set{v_i}_{i=1}^N$, then $A + \delta  \sum_1^N v_i \otimes u_i $ is invertible for every $\delta > 0$, so that $0 \notin \Spec(A + \delta \sum_1^N v_i \otimes u_i )$ and thus spectrum of the perturbation of $A$ is discrete in $D$.}}.
The main result of this paper proves that the second case is generic (in a certain sense\edit{, but unlike the ad hoc perturbation}), so that the stability in Question \ref{qu:stab} fails dramatically.

As a random perturbation, we consider quasi-finite-rank random perturbations of the matrix elements of the operator. 
Although often not indicative of a physical noise profile, it is motivated by numerical algorithms, where one often considers finite-rank approximations of the full operator. 
Thus, a take on Question \ref{qu:stab} from the perspective of a numerical analyst could be whether a noisy finite-rank implementation of a Fredholm operator of index $0$ can still be expected to show traces of non-discrete spectra? 
As our main theorem and the illustration in Figure \ref{fig:1} show, the answer is negative. 

\subsection{The motivating example: Twisted bilayer graphene}\label{sec:motiv}

A representative of the family of operators that motivated {this paper} appears in the mathematical study of \textit{twisted bilayer graphene} (TBG) in the so-called chiral limit, see \cite{Becker2020} for details. The operator is defined as
\begin{align}
D_h(\beta) \coloneqq \mat{2h D_{\bar z } &\beta  U(z) \\ \beta  U(-z) & 2h D_{\bar z}}\colon H^1 (\C /\Gamma ; \C^2 ) \to L^2 (\C /\Gamma ; \C^2) \label{eq:first define of Dh}
\end{align}
where $\beta \in \C$, $\vert \beta\vert = 1$ is the coupling parameter, $D_{\bar z} \coloneq (2i)^{-1}(\p_{\Re z} + i \p_{\Im z})$, 
\begin{align}
U(z) \coloneqq \sum _{k=0}^2  \omega^k e^{\frac{1}{2} (z \bar \omega ^k - \bar z \omega^k)},
\end{align}
$\omega\coloneqq e^{2\pi i /3}$, $\Gamma \coloneq 4\pi (i\omega \Z \oplus i \omega^2 \Z)$, and $h\in \R_{> 0 }$ is proportional to the twisting angle between two stacked layers of graphene. It is proven in \cite{Becker2020} that there exists a discrete set $\mathcal{M} \subset \C$ such that
\begin{align}
\label{eq:M}
\Spec_{L^2 (\C  /\Gamma)} D_h(\beta) = \begin{cases}
\C & \beta / h \in \mathcal{M},\\
\Gamma ^* & \beta / h \notin \mathcal{M}
\end{cases}
\end{align}
where $\Gamma^*$, the dual lattice of $\Gamma$, is a discrete set.\footnote{\edit{
The dual lattice of $\Gamma$ is by definition all $k\in \C $ such that $(\gamma \bar k + \bar \gamma k)\in 4\pi \Z$ for all $\gamma \in \Gamma$, which can be explicitly computed as $\Gamma ^*  = 3^{-1/2} (\omega \Z \oplus \omega^2 \Z )$.}
}
In the theory of twisted bilayer graphene, for fixed $\beta$, the twisting angle $h$ is called \emph{magic} if and only if $\Spec(D_h(\beta))=\CC$. 
\edit{
The set $\mathcal{M}$ has no explicit description, however it can be described as the set of eigenvalues of a certain auxiliary compact operator (see \cite[Theorem 2]{Becker2020} for details).}

A consequence of our main result (Theorem \ref{thm:general result}) is that $D_h(\beta)$ does not exhibit any magic angles {(with overwhelming probability)} if perturbed by a suitable small random perturbation, as described below.

\begin{theorem}[Application to TBG]  \label{Thm.1}
Suppose $\chi(z,\zeta) \in C^\infty_0 (T^* (\C /\Gamma); [0,1])$ is identically $1$ for $|\zeta| < C$, for some sufficiently large $C$, and
\begin{align}
Q_\omega := \mat{\Op_h (\chi) & 0 \\ 0 &\Op_h (\chi) } \circ \left ( \sum _{j,k} \alpha _{j,k} e_j \otimes e_k \right ) \circ \mat{\Op_h (\chi)& 0 \\ 0 & \Op_h (\chi) } 
\end{align}
where $\set{e_j}_{j\in \N}$ is an orthonormal basis of $L^2 (\C / \Gamma ; \C^2)$, $\alpha _{j,k}$ are i.i.d. complex Gaussian random variables with mean $0$ and variance $1$, and $\Op_h (\chi)$ is the Weyl quantization of $\chi$ (see Definition \ref{def:quantization_thing}). 
Then \edit{for fixed $\beta$,} if $0 < \delta < h^{\kappa}$, with $\kappa > 2$, $D_h (\beta) + \delta Q_\omega$ has discrete spectrum with probability at least
\begin{align}
1 - C_1 e^{-C_2 / h ^{2 \kappa }} 
\end{align}
for positive constants $C_1$ and $C_2$ (i.e. with overwhelming probability).
\end{theorem}

In Figures \ref{fig:2} and \ref{fig:1} we provide numerical evidence supporting Theorem \ref{Thm.1}.
In Figure \ref{fig:2}, we compute the spectrum of a finite matrix truncation of $D_h(\beta)$ at a fixed magic angle (top left plot).
The accumulation of eigenvalues near the origin should be interpreted as a finite-rank analogue of having the entire complex plane as spectrum.
We then compute the spectrum at the same $h$ and $\beta$, but with a small random perturbation.
We observe that there is no accumulation of eigenvalues near the origin, suggesting that the magic angles have been washed out by randomness.

In Figure \ref{fig:1}, we attempt to quantify this finite-rank analogue of having the entire complex plane as spectrum.
We fix $\beta = 1$, vary $h$ between $0$ and $2.5$, and measure the density of eigenvalues near the origin.
Any spikes in the eigenvalues should correspond to magic angles.
We then notice two spikes (corresponding to the first two magic angles) as shown in the first plot.
We then do the same computation, but with a random perturbation added, and observe no spikes in eigenvalues.

{
Spectral analysis of random perturbations of non-self-adjoint operators is a rich field within the random matrix theory community.
Indeed, Davies and Hager described the spectrum of randomly perturbed Jordan matrices \cite{davies2009perturbations}.
Similarly, Guionnet, Wood, and Zeitouni described the empirical measure of eigenvalues of certain randomly perturbed non-self-adjoint operators \cite{guionnet2014convergence}.
Basak, Paquette, and Zeitouni considered random perturbations of banded and twisted Toeplitz matrices \cite{basak2019regularization,basak2020spectrum}.
Sj\"ostrand and the author further described similar spectral properties for Toeplitz matrices and Toeplitz matrices \cite{sjostrand2016large,sjostrand2021general}.
}
{In various settings, the limiting spectral measure of randomly perturbed non-self-adjoint operators have been described. See \cite{hager20062} who considered perturbations of $hD_x + g(x)$, \cite{Vogel2020} who considered perturbations of quantizations of tori, and \cite{oltman2023} who considered perturbations of quantizations of K\"ahler manifolds}\edit{.}

\begin{figure}
    \includegraphics[width=6cm]{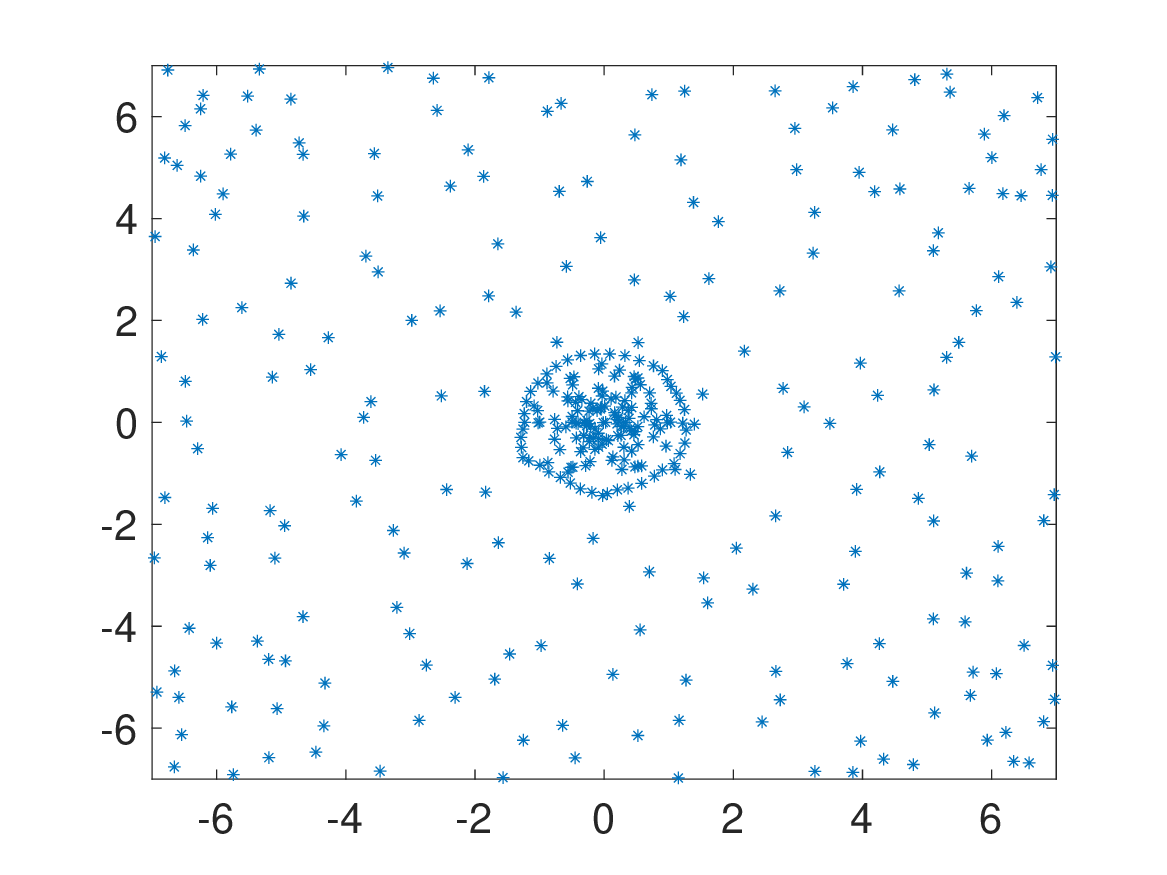}
    \includegraphics[width=6cm]{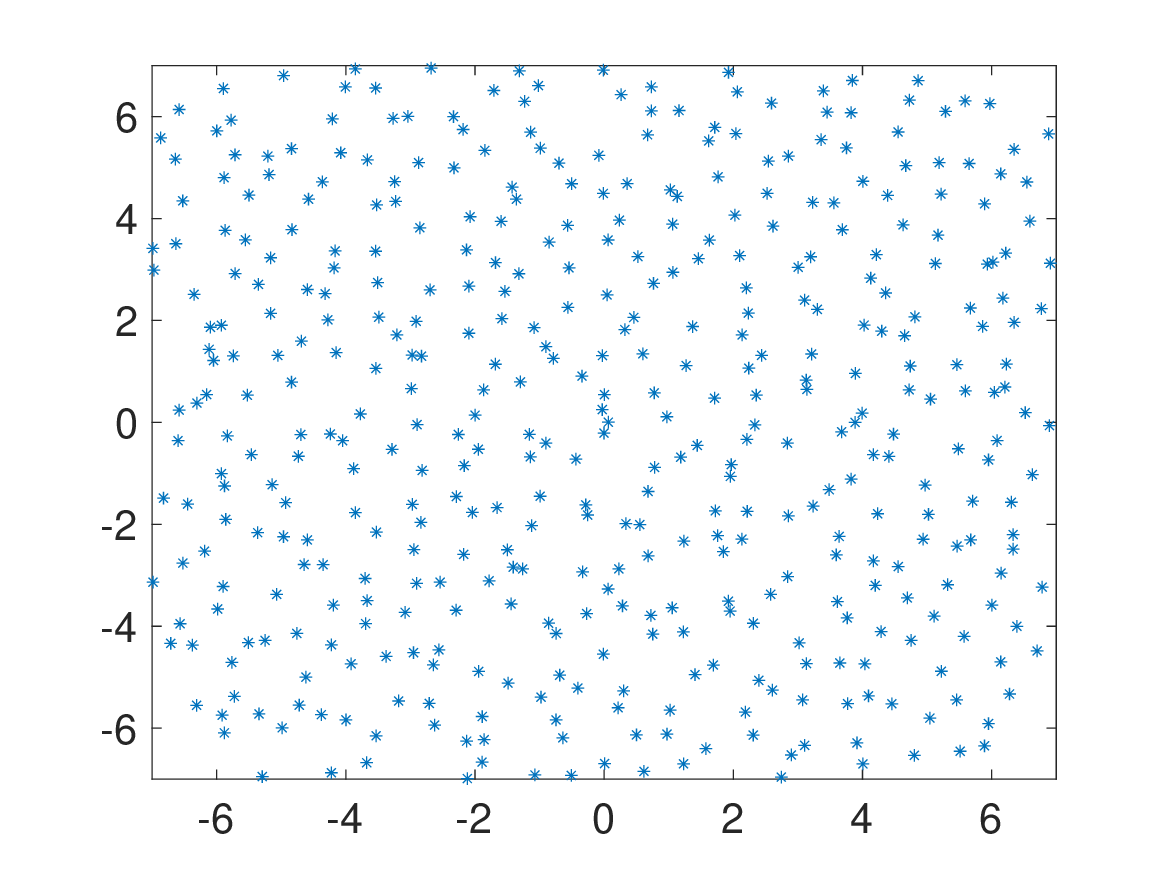}\\
        \includegraphics[width=6cm]{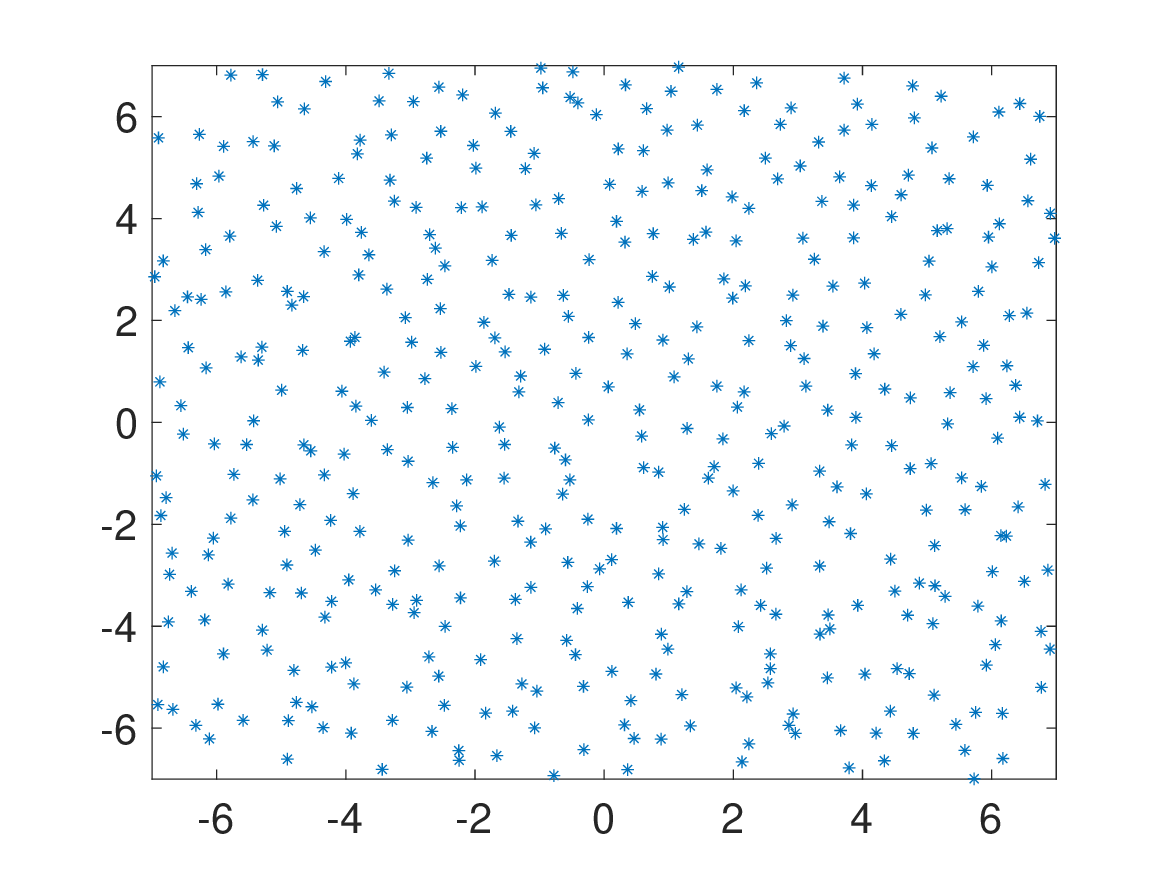}
    \includegraphics[width=6cm]{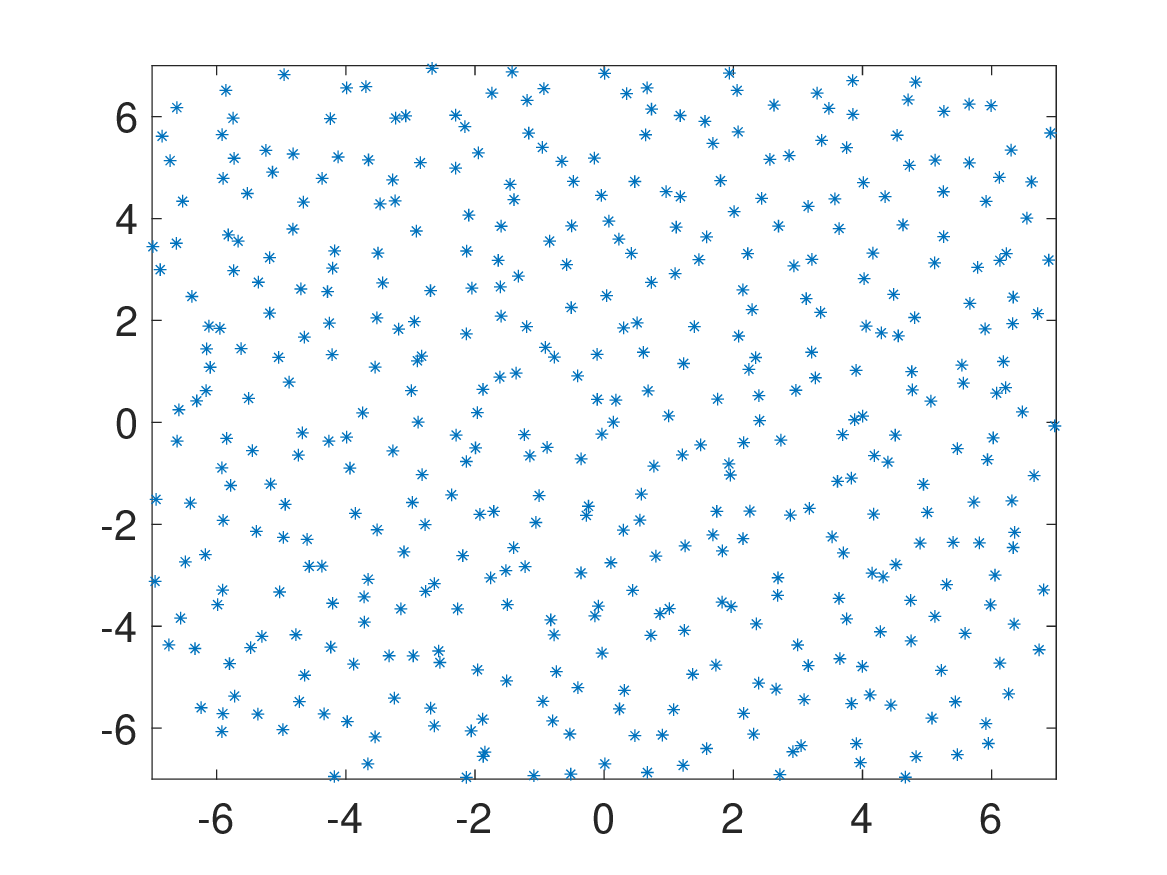}
    \caption{Spectrum of smallest 600 eigenvalues of finite matrix truncation of $D_h(\beta)$, with matrix size 13122, for largest magic $h$ and $\beta=1$ (top left). The accumulation of eigenvalues since $\Spec(D_h(\beta))=\CC$ in the center is clearly visible. Spectrum of finite matrix truncation of $D_h(\beta)$, with same $h,\beta$ and random $\delta =0, 0.01, 10^{-4}, 10^{-7}$ perturbation (clockwise). The accumulation of eigenvalues in the center gets resolved immediately.}\label{fig:2}
\end{figure}

\begin{figure}
\includegraphics[width=5cm]{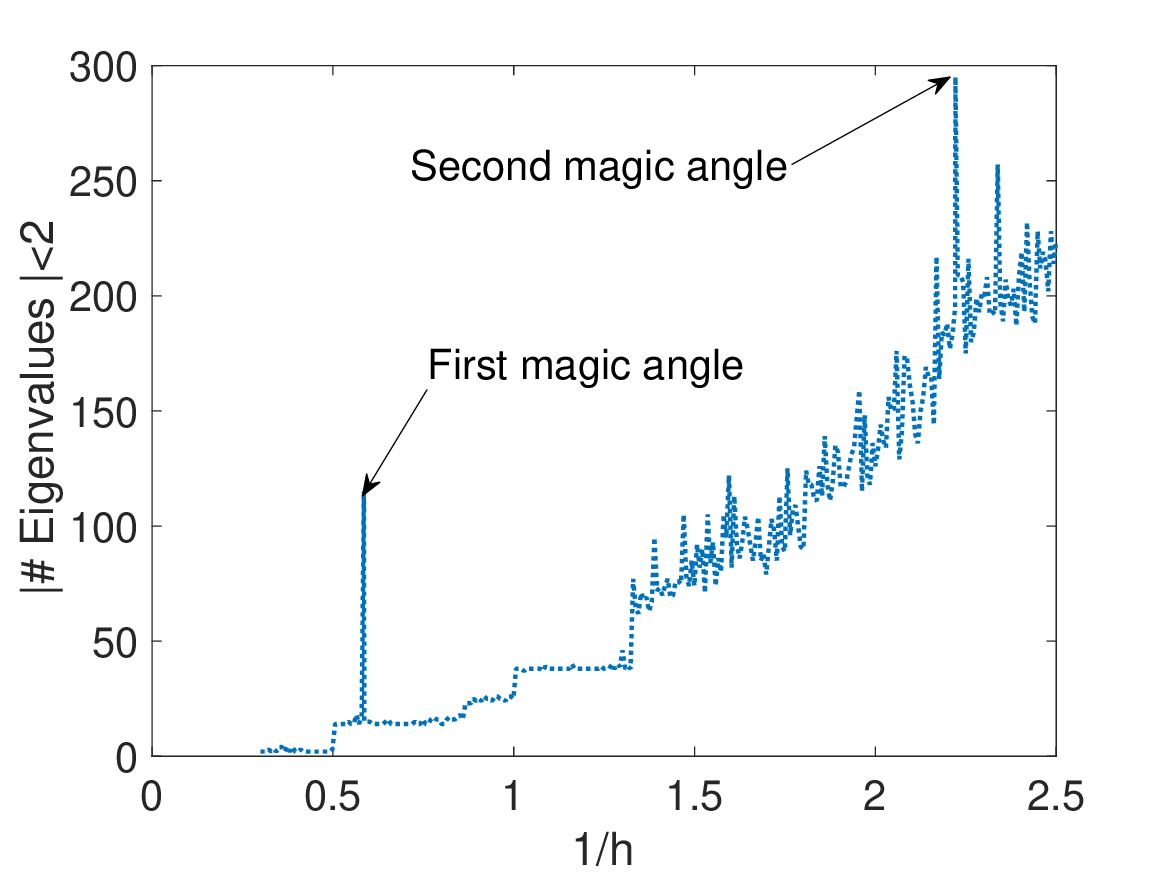}
     \includegraphics[width=5cm]{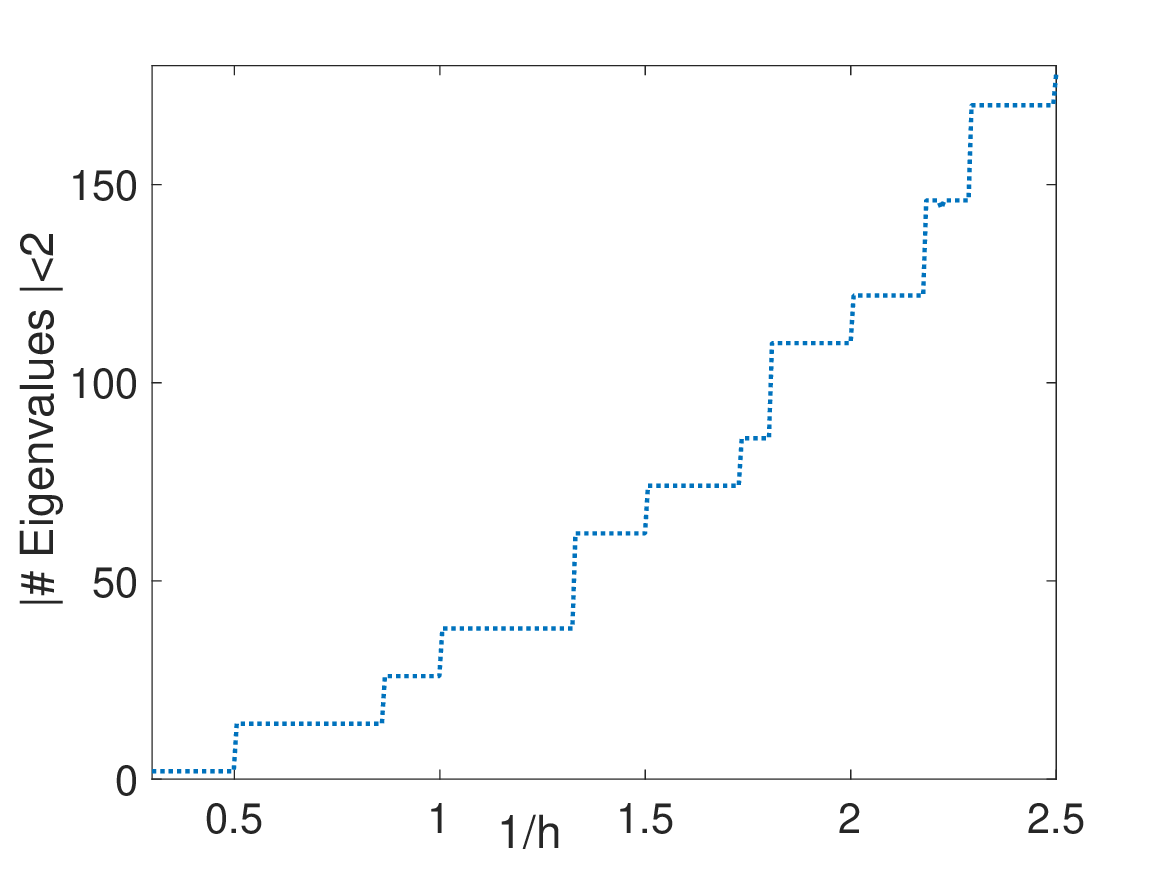}
         \includegraphics[width=5cm]{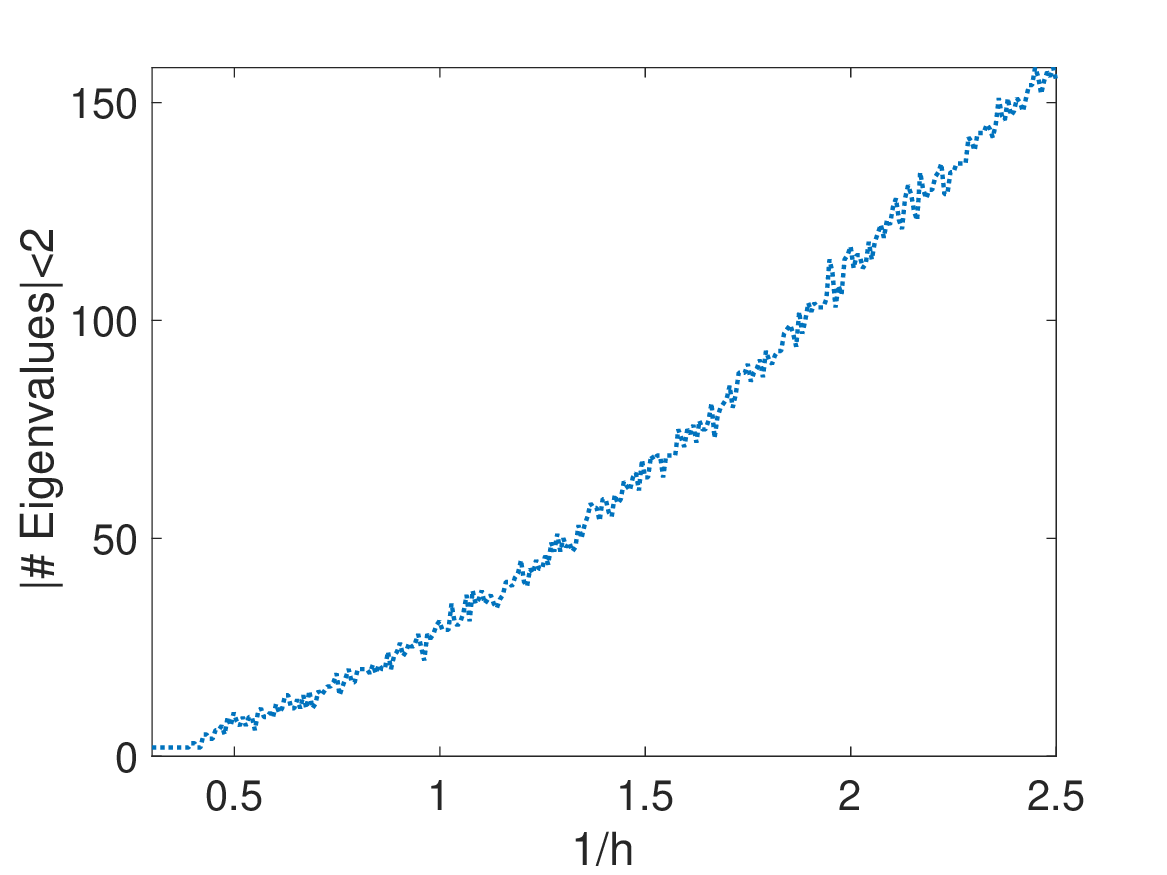}
     \caption{\label{fig:1}Counting the number of eigenvalues of a finite rank approximation of $D_h(1)$ (defined in \eqref{eq:first define of Dh}) for different $h$ in a ball of radius $2$ with random perturbations of size $\delta=0,10^{-5},10^{-1}$ from left to right. Magic angles at $1/h \approx 0.586,2.221$ are clearly visible as spikes in the unperturbed figure (left) but washed out in the perturbed ones (center and right).}
\end{figure}

\smallsection{Outline of paper} 
In Section \ref{section1} we provide the relevant background to state the general result of this paper (Theorem \ref{thm:general result}) as well as a more quantitative result (Theorem \ref{theorem:quantitative}). 
In Section \ref{section2} we provide a proof of these results. 
Then in Section \ref{sec:TBG}, we discuss the application of Theorem \ref{thm:general result} to an operator appearing in the study of magic angles for twisted bilayer graphene, proving Theorem \ref{Thm.1}.
This analysis requires {the already known quantization procedure for matrix-valued functions on $\C / \Gamma$, which we felt worthwhile to outline separately in Appendix \ref{section3}}.

\smallsection{Notation}
For $u,v$ elements of a Hilbert space, the operator $u\otimes v$ operates on $w \in \mathcal{H}$ by:
\begin{align}
(u\otimes v) w = u \ip{w}{v}.
\end{align}
For a function $f$ depending on a positive parameter $h>0$ and $M \in \N$ we write $f = O(h^M)$ if there exist $h_0 > 0$ and $C_M >0$ such that $|f| \le C_M h^M$ for all $0 < h < h_0$.
We write $f= O(h^\infty)$ if $f = O(h^M)$ for all $M \in \N$.

\smallsection{Acknowledgments}
We would like to thank Maciej Zworski for suggesting this problem as well as providing many helpful discussions. 
S. Becker acknowledges support from the SNF Grant PZ00P2 216019. 
I. Oltman was jointly supported by the National Science Foundation Graduate Research Fellowship under grant DGE-1650114 and by the Simons Targeted Grant Award No. 896630. 
M. Vogel was partially funded by the Agence Nationale de la Recherche, through the project ADYCT (ANR-20-CE40-0017).

\section{Setup of problem and main result}\label{section1}

For the remainder of this paper, we let $\mathcal H $ denote a fixed {separable} Hilbert space.
We start with a simple observation about the operators we consider.

\begin{prop}\label{prop1}
Let $A$ be a closed {densely defined} linear operator $A\colon D(A) \subset \mathcal H \to \mathcal H$ and $D\subset \rho_F^{(0)}(A)$ (defined in \eqref{eq:rhoFA0}) be a connected component, then either $D \subset \Spec(A)$ or $D \cap \Spec(A)$ is discrete.
\end{prop}
\begin{proof}
Consider a connected component $D\subset \rho_F^{(0)}(A).$ By assumption, the Fredholm index of $A-z$ is zero for all $z \in D.$ 
\edit{
By analytic Fredholm theory, the Proposition follows.
Indeed, one can follow \cite[Theorem C.8]{dyatlov2019mathematical} which constructs a holomorphic function $f(z)$ with the property $f(z) = 0 \iff$ $z\in \Spec(A)$.
We note that in \cite[Theorem C.8]{dyatlov2019mathematical}, $A$ is assumed to be bounded, however the proof follows without modification for an unbounded $A$.
}
\end{proof}

\edit{
We note that for $j\neq 0$, $\rho_F ^{(j)} (A) \subset \Spec(A)$.
This trivially follows by observing that if $z\in \rho_F ^{(j)}(A)$ (for $j\neq 0$), then the kernel or cokernel of $A-z$ is nontrivial and thus $A-z$ is not bijective.
}

\subsection{Random perturbation}
{Here we discuss the random perturbations considered in this work.}
Let $\set{e^1_j : j \in \N}$ and $\set{e^2_j : j \in \N}$ be two orthonormal bases of $\mathcal{H}$, and define
\begin{align}
Q_\omega  \coloneqq S_1 \circ \left  ( \sum _{j,k=1}^\infty \alpha _{j,k} e^1_j \otimes e^2_k \right )  \circ S_2 \label{random_pertubation}
\end{align}
where $S_1$ and $S_2$ are fixed Hilbert--Schmidt operators and $\alpha _{j,k}$ are independent identically distributed {complex} Gaussian random variables with mean $0$ and variance $1$. To make sense of the sum in \eqref{random_pertubation}, we decompose $S_1$ and $S_2$ using a singular value decomposition and use that the law of $(\alpha _{j,k})$ is invariant under unitary transformations, to rewrite $Q_\omega$ as
\begin{align}
Q_\omega = \sum _{j,k}  s_j^1  s_k ^2 \alpha _{j,k}   e_j^1 \otimes  e_j^2
\end{align}
where $ s_j^{1,2}$ are the singular values\footnote{If $S$ is a Hilbert--Schmidt operator, we say $\set{s_j}_{j\in \N}$ are the singular values of $S$ if $s_j$ are the eigenvalues of $\sqrt{S^* S}$. Although not necessarily here, we assume $s_j$ are decreasing to zero {as $j\to\infty$}.} of $S_1$ and $S_2$ respectively and $ e_j^{1,2}$ are two potentially different (from the ones above) orthonormal bases of $\mathcal{H}$. Each term in this sum is an independent complex Gaussian random variable of mean zero and variance $ s_j ^1  s_k^2$. 
It can be shown that with probability $1$, the Hilbert--Schmidt norm of $Q_\omega$ is finite (this follows from Lemma \ref{lemma.norm of qomega} and continuity of measures). 
And under the additional assumption that $S_1$ or $S_2$ is trace-class (we will assume in our main result that $S_1$ is trace-class), then $Q_\omega$ is almost surely trace-class (see \eqref{eq:trace_bound_qomega}).

To analyze $A + \delta Q_\omega$, we require certain elliptic-type properties of $S_1$ and $S_2$. {
Because $A$ is closed and densely defined, the adjoint of $A- z$ (for any fixed $z\in \C$), denoted by $(A- z)^*$ has dense domain 
\begin{align}
    D((A-z)^*)) \coloneq \set { u \in \mathcal{H} : \exists v \in \mathcal{ H } 
 \text{ such that }  \ip{(A-z) w}{u} = \ip{w}{v} \ \forall \  w \in D((A-z)) },
    \end{align}
    see for instance \cite[Theorem VIII.1]{reed1972methods}.
We may then define the operator $(A-z)(A-z)^*$ with domain
\begin{align}
    D((A-z)(A-z)^*) = \set{u \in D((A-z)^*) : (A-z) u \in D((A-z))}. \label{eq:domain}
\end{align}
A result of von Neumann (see for instance \cite[Theorem 3.24]{kato2013perturbation}) states that $(A-z)(A-z)^*$ with this domain is self-adjoint.
}
For the following hypothesis to make sense, we assume that
\begin{align}
    \exists \tilde z\in \C  \text{ such that } \Spec\left((A - \tilde z ) (A - \tilde z )^*\right) \text{ is discrete}. \label{assump:discrete spectrum}
\end{align}
This assumption can be relaxed to $\Spec\left((A - \tilde z ) (A - \tilde z )^*\right)$ being discrete near $0$.
\edit{We stress here that \eqref{assump:discrete spectrum} is much weaker than $A-\tilde z$ having discrete spectrum (which we aim to prove for a perturbation of $A$). 
See \S \ref{sec:TBG} where we prove \eqref{assump:discrete spectrum} for the operator $D_h(\beta)$ coming from TBG.}

{Our proof relies on constructing an auxiliary operator which has components projecting onto the singular vectors of $(A-z)$, which uses that the spectrum of $(A-z)(A-z)^*$ is discrete.
If we only have that the spectrum of $(A-z)(A-z)^*$ is discrete near zero, we can still build an auxiliary operator by using spectral projectors, at the cost of requiring finer analysis.
We will leave this added generality to future work.}

\begin{hyp}\label{def1}
For $\alpha  > 0 $ and $\tilde z  \in \C$ satisfying \eqref{assump:discrete spectrum}, we denote by $\set{e_1,\dots, e_N}$ and $\set{f_1, \dots ,f_N}$ the eigenvectors of $(A- \tilde z  ) (A- \tilde z )^*$ and $(A-\tilde z )^* (A- \tilde z )$ {(equipped with an analogously defined domain to \eqref{eq:domain})} with eigenvalues less than $\alpha$, respectively.
Let 
\begin{align}
    \mathcal{H}_{1,\alpha} \coloneq \Span(\set{e_i : i = 1, \dots, N}) &&\text{and} && \mathcal{H}_{2,\alpha} \coloneq
\Span(\set{f_i: i = 1,\dots, N}).
\end{align}

Because our main motivation is to apply our result to semiclassical operators, we allow $S_1$ and $S_2$ to depend on a positive parameter $h$.
We assume that for $S_1 = S_1(h,\alpha)$ and $S_2 = S_2(h,\alpha)$ there exists a $C_S = C_S (h,\alpha) > 0 $ such that 
\begin{subequations} \label{CS constant}
    \begin{align}
\norm{S_1 v } \ge C_S \norm{v} && \forall \ v \in \mathcal{H}_{1,\alpha },\\
\norm{S_2 w } \ge C_S \norm{w} && \forall \ w\in \mathcal{H} _{2,\alpha}\label{CS constant2}.
\end{align}
\end{subequations}

The set of pairs of operators $(S_1,S_2)$ that satisfy these properties will be denoted by $\mathfrak{S}(\tilde z, \alpha)$.

\end{hyp}

{We note that for a choice of $\alpha$, the number of eigenvalues of $(A-\tilde z ) (A- \tilde z )^*$ may be zero.
In this case, $\mathcal{H}_{1,\alpha}$ and $\mathcal{H}_{2,\alpha}$ are empty, so we do not require additional assumptions for $S_1$ and $S_2$.
Note that in this case, $\tilde z\notin \Spec(A)$, so before being randomly perturbed, $\Spec(A)$ is discrete within the connected component of $\rho_F^{(0)}(A)$ containing $\tilde z$.
}
{For such an $A$, Theorem \ref{thm:general result} stated below still holds, i.e. that the perturbed operator has discrete spectrum with high probability.}

{We implicitly assumed in Hypothesis \ref{def1} that $N$ is finite.
This will always be true for $A$ strictly unbounded (that is, when $A$ is \textit{not} bounded).}

\subsection{Main result}
We now state the main result of our paper showing that random perturbations {(of the form \eqref{random_pertubation})} of closed linear operators $A\colon D(A) \subset \mathcal H \to \mathcal H$ immediately yield a discrete spectrum with high probability. Because we want to apply this to semiclassical operators that depend on a parameter $h> 0$, we state this result for such operators.

\begin{theorem}[General Result]\label{thm:general result} 
Given the following:
\begin{enumerate}
\item \label{condition:1} Let $A=A(h)$ be a family of closed (possibly) unbounded operators $A\colon D(A) \subset \mathcal H \to \mathcal H$ indexed by a parameter $h \in \R_{>0}$.
\item \label{condition:2} 
\edit{Setting $D_j$ as disjoint, open, connected sets such that $\bigcup_{j \in \mathcal J} D_j \subset \rho_F^{(0)}(A)$ (with $\mathcal{J}$ an at most countable index set), then for} all $j\in \mathcal{J}$, there exists $z_j \in D_j$ such that $(A - z_j)^*(A-z_j)$ has discrete spectrum.
\item There are two families of Hilbert--Schmidt operators $$(S_1(h,\alpha), S_2(h,\alpha)) \in \bigcap_{j\in \mathcal{J}}\mathfrak{S}(z_j,\alpha)$$ (defined in \eqref{random_pertubation} and satisfying Hypothesis \ref{def1}) with constant $C_S  = C_S (h,\alpha ) > 0$ (as in \eqref{CS constant}).
\item \label{condition:5} For each $\alpha$ and $h$, $S_1(h,\alpha)$ is a trace-class operator.
\end{enumerate}
Then there exists $C_0> 0 $ such that $\Spec(A + \delta  Q_\omega) \cap \left( \bigcup_{j\in \mathcal{J}} D_j \right)$ (for $\delta > 0$ and $Q_\omega$ defined in \eqref{random_pertubation}) is discrete with probability at least
\begin{align}
\max \left ( 1 - C_0 \exp \left ( \frac{C_0 \norm{S_1}\norm{S_1}_{\operatorname{Tr}} \norm{S_2}_{\operatorname{HS}}^2 - \alpha \delta^{-2}}{2\norm{S_1}^{3/2} \norm{S_2}} \right ), 0 \right ) \label{eq:prob bound}
\end{align}
where $\norm{\cdot}_{\operatorname{Tr}}$ denotes the trace norm.
\end{theorem}

We have a few remarks about Theorem \ref{thm:general result}.

Assumption \ref{condition:5} can be replaced with $S_2(h,\alpha)$ being trace-class, {in which case in \eqref{eq:prob bound} all norms of $S_1$ are replaced by the same norms of $S_2$ and vice versa.}

The probabilistic bound \eqref{eq:prob bound} is close to $1$  if
\begin{align}
\frac{C_0 \norm{S_1}\norm{S_1}_{\operatorname{Tr}} \norm{S_2}_{\operatorname{HS}}^2 - \alpha \delta^{-2}}{2\norm{S_1}^{3/2} \norm{S_2}} 
\end{align}
is a large negative number. This can be done by increasing $\alpha$ and/or decreasing $\delta$. However, if $\alpha$ is large, the operators $S_1$ and $S_2$ may need to be modified to satisfy Hypothesis \ref{def1}, which may increase $\norm{S_1}_{\operatorname{Tr}}$ and $\norm{S_2}_{\operatorname{HS}}$.

Alternatively, we can fix $\alpha$ and construct $S_1$ and $S_2$ to have norm $1$, but at the cost of having large Hilbert--Schmidt and trace norms. 
This forces us to select a small $\delta$. 

One could artificially satisfy Hypothesis \ref{def1} by setting $S_1$ and $S_2$ to be the orthogonal projections onto ${\rm{span}}(\set{e_1,\dots, e_N})$ and ${\rm{span}} (\set{f_1,\dots, f_N})$ respectively.
In this case, $\norm{ S_1 } = \norm {S_2} = 1$ and $\norm{S_1}_{\operatorname{Tr}} = \norm{S_2}_{\operatorname{HS}}^2 = N$. 
We would then find that $A + \delta Q_\omega$ has a discrete spectrum with probability close to $1$ as long as $\delta \ll N^{-1}\sqrt{\alpha}$. 
Understanding how $N$ depends on $\alpha$ on a case-by-case basis would further refine this estimate.

We furthermore note that Theorem \ref{thm:general result} applies to unbounded closed operators such that $\rho_F^{(0)}(A)$ is non-empty (i.e. operators $A$ such that there is a $z_0 \in \C$ such that $A- z_0$ is Fredholm of index $0$) -- not just Fredholm operators of index $0$. 

We also remark that for the previous discussed operators (Seeley's operator \eqref{eq:seeleyoperator} and $D_h(\beta)$ \eqref{eq:first define of Dh}), $\rho_F^{(0)}(A) = \rho_F ^{(0)}(D_h(\beta)) = \C$.
Operators with several connected components of $\rho_F^{(0)}(A)$ do exist. For example, Toeplitz operators on the classical Hardy space $H^2 (\T)$ have Fredholm index related to the winding number of the image of the circle under the symbol \cite[Theorem 7.26]{douglas2012banach}. 

The assumption \ref{condition:2} in Theorem \ref{thm:general result} will be satisfied for $D_h(\beta)$ by using the ellipticity {at infinity} of the operator and the compact embedding theorems of Sobolev spaces.

We next state a more quantitative version of Theorem \ref{thm:general result}.
Instead of proving $A + \delta Q_\omega- z_0$ is invertible with certain probability (which would imply a discrete spectrum), we show that the smallest singular value cannot be too small (which implies invertibility).

\begin{theorem}\label{theorem:quantitative}
Given the same assumptions as Theorem \ref{thm:general result}, for all $\delta >0$ and $z_0 \in D_j$ (for some $j\in \mathcal{J}$), \edit{in the event that $\norm{Q_\omega}_{HS}<\infty$ (which occurs with probability $1$)
}
let $0\le t_{1,\delta} \le t_{2,\delta} \le \cdots$ denote the singular values\footnote{
\edit{
The fact that the singular values are discrete is not immediate and is discussed after the statement of the theorem.
}
}
of $A+\delta Q_\omega - z_0$, and fix $\alpha_{\rm{max}} >0$.
Then there exist constants $C_0,C_1,C_2> 0 $ such that for $a > 0$ satisfying:
    \begin{align}
        a \le N^{NC_2} \delta ^N \alpha^{-(N-2)/2} \label{eq:570}
    \end{align}
    we have that
    \begin{align}
        \P(t_{1,\delta} \le a) \le C_0 \exp\left( \frac{C_0 \norm{S_1}\norm{S_1}_{\operatorname{Tr}} \norm{S_2}_{\operatorname{HS}}^2 -\alpha \delta ^{-2}}{2 \norm{ S_1}^{3/2} \norm{S_2}} \right) +   \frac{C_0\alpha ^{2 + (N-1)/2}}{N^{C_1N}\delta ^{N-2}C_{S}^{2N}  \norm{S_{1}}^{3/2}}a \label{eq:574}
    \end{align}
for $\alpha < \alpha _{\rm{max}}$ where {$C_S$ and $N(\alpha)$ are defined in Hypothesis \ref{def1}.}
\end{theorem}

\edit{
Discreteness of the singular values of $A+\delta Q_\omega -z_0$ follows from the fact that the essential spectrum is invariant under compact perturbations.
Indeed, let $\mathcal A:=(A-z_0)^*(A-z_0)$ and $\mathcal B:=(A-z_0 +\delta Q_\omega)^* (A-z_0 +\delta Q_\omega).$
If we can show that
\begin{equation}\label{eq:587}
\begin{split}
    &(\mathcal A+i)^{-1}-(\mathcal B+i)^{-1}\\
    &\qquad\qquad\qquad=(\mathcal B+i)^{-1} (\delta (Q_{\omega}^*(A-z_0)+(A-z_0)^*Q_{\omega} + \delta Q_{\omega}^*Q_{\omega}))(\mathcal A+i)^{-1} 
    \end{split}
\end{equation}
is compact, then because $\mathcal A$ has only discrete spectrum, then $\mathcal B$ has discrete spectrum (see \cite[\S XIII.4, Corollary 1]{reed1978iv}).
Compactness of \eqref{eq:587} follows because $(A-z_0)(\mathcal{A}+i)^{-1}$ and $(\mathcal{B}+i)^{-1} (A-z_0)^*$ are bounded (the first can be shown by taking the polar decomposition of $A$ and using functional calculus, the second by the Kato--Rellich theorem) and $Q_\omega$ is compact. 
}

{
We note that applications of Theorem \ref{theorem:quantitative} should be for $a$ close to zero, and the required upper bound on $a$ \eqref{eq:570} is a technicality needed to apply a random matrix theory result.
Heuristically, Theorem \ref{theorem:quantitative} states that as $a$ goes to zero, the probability the smallest singular value of $A+ \delta Q_\omega - z_0$ is less than $a$ goes to zero.
Indeed, the first term on the right-hand side of \eqref{eq:574} is the same term appearing in \eqref{eq:prob bound}, which is independent of $a$.
In the case of TBG, this term is exponentially small in $h$ (see Theorem \ref{theorem_pert_of_DhB}).

The second term must also be treated on a case-by-case basis, in particular because it depends on $N$ (which depends on $\alpha$ and $h$).
In the case of TBG, if we fix $\alpha$, let $\delta = h^{4}$, and construct $S_1$ and $S_2$ as in Theorem \ref{theorem_pert_of_DhB}, we can apply \cite[Proposition 4.2]{becker2024absence} which gives us that $N \asymp h^{-2}$.
In this case, there exist positive constants $C_3$ and $C_4$ such that the second term on the right-hand side of \eqref{eq:574} is bounded above by
\begin{align}
    \frac{ a h^8}{(C_3 h^2)^{h^{-2}C_4}}
\end{align}
which rapidly goes to zero if either $a$ or $h$ goes to zero.
}

{
Estimating the size of the smallest singular value of a randomly perturbed operator has been well-studied in the context of random matrix theory.
Sankar, Spielmann, and Teng in \cite{sankar2006smoothed} estimate the probability that the smallest singular value of $X + \delta Q$ is small for $X$ a deterministic matrix $N\times N$, $Q$ a random matrix $N\times N$ with Gaussian i.i.d. entries, and $\delta >0$.
There have been various extensions or related work, see for instance Rudelson and Vershynin \cite{rudelson2008least} who estimate the size of the smallest singular value of a matrix whose entries are i.i.d. subgaussians, Tao and Vu \cite{tao2008random,tao2009smooth} where $Q$ has i.i.d. entries of bounded second moment, and Cook \cite{cook2018lower} where $Q$ has independent but not identically distributed entries under additional assumptions.
See also Basak, Paquette, and Zeitouni \cite[Remark 1.3]{basak2020spectrum} who describe random perturbations $Q$ such that the smallest singular value of $X + \delta Q$ is small.
}

\section{Proof of main result}\label{section2}

We begin by stating a lemma that provides a probabilistic norm bound of $Q_\omega$ as defined in \eqref{random_pertubation}, proven in \cite[Remark 6.2]{Hager2008} and \cite[Equation 8.5]{Hager2008}.
\begin{lemma}\label{lemma.norm of qomega}
There exists a $C_0 > 0$ such that for all $a >0$:
\begin{align}
\P (\norm{Q_\omega}_{\operatorname{HS}}^2 \ge a) \le C_0 \exp\left ( \frac{C_0 \norm{S_1}_{\operatorname{HS}}^2 \norm{S_2}_{\operatorname{HS}}^2 -a }{2\norm{S_1} \norm{S_2}} \right ).
\end{align}
\end{lemma}
Under the assumption that $S_1$ is trace-class, we can write it as a product of Hilbert--Schmidt operators whose Hilbert--Schmidt norms are both the square root of the trace norm of $S_1$ (this is further expanded in the proof of Theorem \ref{thm:general result}{)}.
We can therefore deduce from Lemma \ref{lemma.norm of qomega} a bound on the trace norm of $Q_\omega$:
\begin{align}
    \P(\norm{Q_\omega}_{\operatorname{Tr}}^2 \ge a ) \le C_0 \exp\left ( \frac{C_0 \norm{S_1}_{\operatorname{Tr}}\norm{S_2}_{\operatorname{HS}}^2 - a \norm{S_1}_{\operatorname{Tr}}^{-1}}{2\norm{S_1}^{1/2} \norm{S_2}} \right ).\label{eq:trace_bound_qomega}
\end{align}
{By continuity of measures, we get that almost surely $Q_\omega$ has finite trace norm.}

Equation \eqref{eq:trace_bound_qomega} follows by decomposing $S_1$ as $S_1 = S_{1,1}S_{1,2}$ (for appropriately chosen $S_{1,1},S_{1,2}$ as discussed in the proof of Theorem \ref{thm:general result}), and using that
\begin{align}
    \norm{Q_\omega}_{\operatorname{Tr}} \le \norm{S_{1,1}}_{\operatorname{HS}}\norm{S_{1,2}\circ \left( \sum_{j,k=1}^\infty \alpha_{j,k} e_j^1\otimes e_k^2 \right)\circ S_2} _{\operatorname{HS}}.
\end{align}

We now prove Theorem \ref{thm:general result}. 
This proof follows the method of proof of \cite{Hager2008}. 
However, the main difference is that our operators do not necessarily have semiclassically elliptic symbols.
\begin{proof}[Proof of Theorem \ref{thm:general result}]
We begin by selecting one $j\in \mathcal{J}$, setting $D \coloneqq D_j$ and $z_0 \coloneqq z_j$ given in the hypothesis. Next set $A^\delta \coloneq  A + \delta Q_\omega$. Because $Q_\omega$ is compact, and compact perturbations of Fredholm operators are Fredholm operators with the same index, $A^\delta - z$ is Fredholm of index zero for all $z\in D$. Therefore, by Proposition \ref{prop1}, the spectrum of $A^\delta$ either contains a connected component $D \subset \rho_F^{(0)}(A)$ or $D \cap \Spec(A)$ is discrete. 

Our overall goal is to show (with appropriate probability) that $z_0$ is not in the spectrum of $A^\delta$, thus proving $D \cap \Spec(A^\delta)$ is discrete.
To achieve this, we will build an {auxiliary operator called} a Grushin problem. 

\noindent \textbf{Step 1. Build a Grushin problem for the unperturbed operator} 

{{Equip $ \SAA \coloneq (A-z_0)^*(A-z_0)$} with the the {analogously defined} domain to \eqref{eq:domain} and {$\tilde \SAA \coloneq (A - z_0) (A-z_0)^*$} with domain \eqref{eq:domain} such that both operators are self-adjoint.}
Let $0 \le t_1^2 \le t_2^2 \le \cdots$ be the eigenvalues of $\SAA$ with {an} orthonormal {basis of} eigenvectors $\set{e_i}_{i\in \N}$ such that $t_i \in [0,\infty)$.
Suppose $t_1= t_2 = \cdots = t_n =0$ {and $t_{n+1} > 0$}. 
Then because {$(A-z_0)$} is Fredholm of index $0$, and $\operatorname{Ker}(\Theta)= \operatorname{Ker}(A-z_0)$, the dimension of the {kernel} of {$(A-z_0)^*$} has dimension $n$.
Let $f_1,\dots f_n$ be an orthonormal basis for this space, so that $\tilde \SAA f_i = 0 = t_i^2 f_i$.
If $t_i \neq 0$, then define
\begin{align}
    f_i \coloneq \frac{1}{t_i} (A - z_0)e_i.
\end{align}
Therefore:
\begin{align}
    \tilde \SAA f_i = (A - z_0) (A- z_0)^* f_i = \frac{1}{t_i} (A -z_0) (A-z_0)^*(A- z_0) e_i = t_i (A- z_0)e_i = t_i^2 f_i.
\end{align}

We then get that $\set{f_i}_{i\in \N}$ {is an} orthonormal {system of} eigenvectors of $\tilde \SAA$ corresponding to the same eigenvalues of $\SAA$ such that:
\begin{align}
(A - z_0) e_i = t_i f_i  && \text{and} && (A - z_0)^* f_i = t_i e_i
\end{align}
for all $i\in \N$. 
{
We moreover claim that $\set{f_i}_{i\in \N}$ is an orthonormal basis of eigenvectors of $\tilde \SAA$ (which implies $\SAA$ and $\tilde \SAA$ have the same eigenvalues).
Indeed, suppose {$f\in D((A-z_0)^*)$ is such that $\ip{f}{f_i} = 0$ for all $i\in \N$, it suffices to show that $f = 0$.
We have that $\ip{(A-z_0)^* f}{e_i}= \ip{f}{(A-z_0)e_i}$.
If $i \le n$, then $(A-z_0)e_i = 0$ and if $i >n$, then $\ip{f}{(A-z_0)e_i} = t_i \ip{f}{f_i} = 0$. 
Therefore $\ip{(A-z_0)^* f}{e_i}= 0$, so that $f \in \operatorname{Ker}((A-z_0)^*)= \operatorname{Ker}\tilde \SAA$.
This implies $f=0$ because $f_1,\dots,f_n$ span $\operatorname{Ker}\tilde \SAA$.
{Because $D((A-z_0)^*)$ is dense in $\mathcal H$, we get by approximation that if $\ip{f}{f_i} = 0$ for all $i\in \N$, then $f =0$.}
}}

Let $\alpha > 0$, and define:
\begin{align}
N  = \min \set{ j\in \Z : t_j^2 > \alpha}-1. \label{eq:694}
\end{align}
{For now, we will assume that $N>0$, the case when $N=0$ will be treated at the end of this proof.}
Now define $R_- = \sum _1^N f_i \otimes \delta _i$ and $R_+ = \sum _1^N \delta _i \otimes e_i$ (where $\delta_i$ is the standard basis of $\C^N$), so that
\begin{align}
\mathcal{P}\coloneqq \mat{A - z_0 & R_- \\ R_+ & 0 } \colon {D(A)} \times \C^N \to \mathcal{H} \times \C^N
\end{align}
is invertible with inverse
\begin{align}
\mathcal{E} \coloneqq \mat{\sum _{N +1}^{\infty} \frac{1}{t_i} e_i \otimes f_i && \sum_1^N e_i \otimes \delta _i\\
\sum_1^N \delta _i \otimes f_i && - \sum_1^N t_i \delta _i \otimes \delta _i } \coloneq \mat{ E^0 & E_+^0 \\ E^0_-  & E_{-+}^0}\label{eq:656}
\end{align}
as in {\cite[\S 4]{Sjo2009}}. 
Using \eqref{eq:694}, we can bound (or explicitly compute) the norms of the block entries of $\mathcal{E}$ as
\begin{align}\label{eq:770}
\begin{array}{cc}
       \norm{E^0} \le \frac{1}{\sqrt{\alpha}}, & \norm{E_+^0} = 1,\\
    \norm{E_-^0} = 1, & \norm{E_{-+}^0} \le \sqrt{\alpha}.
\end{array}
\end{align}

\noindent \textbf{Step 2. Build a Grushin problem for the perturbed operator}

We now define
\begin{align}
\mathcal{P}^\delta \coloneqq \mat{A^\delta - z_0 & R_- \\ R_+ & 0 } \colon D(A) \times \C^N \to \mathcal{H} \times \C^N 
\end{align}
so that
\begin{align}
    \mathcal{P}^\delta \mathcal{E} = 1 + \mat{\delta Q_\omega E^0 & \delta Q_\omega E_+^0 \\ 0 & 0 } \coloneq 1 + K.
\end{align}
The term $1+K$ can be inverted via a Neumann series provided that $\norm{K} < 1 $, which is satisfied if
\begin{align}
\delta \alpha^{-1/2} \norm{Q_\omega} {< (1-\e_0)} < 1 \label{eq:condition on Q_omega}
\end{align}
{for some fixed $\e_0 \in (0,1)$.}
When this is the case, {$\mathcal{P}^\delta$ is bijective with inverse}
\begin{align}
\mathcal{E}^\delta\coloneqq (\mathcal{P}^\delta )^{-1}  = \mat{E^\delta & E_+ ^\delta \\ E_- ^\delta & E_{-+}^\delta}.
\end{align}
By the Schur complement formula, $z_0 \notin \Spec A^\delta$ if and only if $E_{-+}^\delta$ is invertible. By construction, $E_{-+}^\delta$ is an $N\times N$ matrix. It now suffices to estimate the probability $\det E_{-+}^\delta$ is nonzero.

By the Neumann series construction, the terms in $\mathcal E ^\delta $ can be computed by writing $\mathcal{E}^\delta = \mathcal{E} (1+K)^{-1} = \mathcal{E}\sum_0^\infty( -K)^j$. 
In this case, we get that
\begin{align}
E_{-+}^\delta= E^0_{-+} + \delta E^0_- Q_\omega E^0_+ + \sum _{j=2}^\infty E^0_- (\delta Q_\omega E^0)^{j-1} \delta Q_\omega E^0_+. \label{eq:607}
\end{align}

Define
\begin{align}
T \coloneqq \delta ^{-1}\left  ( \sum _{j=2}^\infty E^0_- (\delta Q_\omega E^0)^{j-1} \delta Q_\omega E^0_+\right ) , && \hat Q_\omega \coloneqq  E^0_- Q_\omega E^0_+  + T, \label{eq:313}
\end{align}
so that
\begin{align}
E_{-+}^\delta =   E^0_{-+} + \delta  \hat Q_\omega. \label{first Epm}
\end{align}
Note that there exists a $C> 0$ such that
\begin{align}
\norm{T}_{\operatorname{HS}}\le C \frac{\delta}{\sqrt{\alpha}} \norm{Q_\omega}_{\operatorname{HS}}^2 \label{eq:312}
\end{align}
by \eqref{eq:770}, {\eqref{eq:condition on Q_omega}}, and \eqref{eq:313}, and the fact that the Hilbert--Schmidt norm of $T$ is bounded by a constant times the norm of the first term in its summation expansion in \eqref{eq:313}.

\noindent \textbf{Step 3. Describe the law of $\hat Q_\omega$} 

{We reiterate that $z_0 \notin \Spec A^\delta$ if and only if $E_{-+}^\delta$ is invertible.
The remainder of the proof will show that $E_{-+}^\delta$ is invertible with high probability. 
By \eqref{first Epm}, $E_{-+}^\delta$ can be written as a sum of a deterministic $N\times N$ matrix and $\delta$ times a random $N\times N$ matrix $\hat Q_\omega$.
By \eqref{eq:313}, $\hat Q_\omega$ has a leading order term (in $\delta$) which we will describe in the following proposition.}

\begin{prop}\label{claim.1}
There exist orthonormal bases $\set{f^1_j} _{j = 1}^{N}$ and $\set{ f^2_j} _{j = 1}^{N}$ of $\C^N$ such that
\begin{align}
 E^0_- Q_\omega E^0_+  = \sum _{1 \le j,k\le N} s_{1,j} s_{2,k} \alpha _{j,k} f^1_j \otimes f^2_k
\end{align}
where $s_{1,j},s_{2,k}$ are such that ${0<}C_S \le s_{1,j},s_{2,k}\le \max (\norm{S_1}, \norm{S_2})$.
\end{prop}
\begin{proof}
For a trace-class Fredholm operator {$B$} of index $0$, {its singular values are denoted by {$s_1(B)\ge s_2 (B)\ge \cdots$}. 
These are the eigenvalues of {$(B^*B)^{1/2}$} ordered in decreasing values and counting their multiplicities.}

With this notation, let $s_{1,j}$ denote the singular values of $E^0_- \circ S_1$ and let $s_{2,j}$ denote the singular values of $S_2 \circ E^0_+$. 
Note that $E^0_- \circ S_1$ and $S_2 \circ E^0_+$ are both rank $N$ operators, so that {$s_{2,j} = 0$ for $j \ge N+1$ and $E_-^0 \circ S_1$ has only $N$ singular values}. 
Both operators are bounded, so $s_{1,j}, s_{2,j} \le \max (\|S_1\|, \|S_2\|)$.

By Hypothesis \ref{def1}, we get a lower bound for the first $N$ singular values.
Indeed, if $(S_2 E_+^0)^*(S_2 E_+^0) u = (s_{2,N} )^2 u$ with $\norm{u} = 1$, then
\begin{align}
(s_{2,N})^2 = \ip{(s_{2,N})^2  u}{u} = \norm{S_2 E_+^0 u} ^2 \ge C_S ^2 \norm{E_+^0 u } ^2 = C_S^2 .
\end{align}
{The last equality uses that $E_+^0 = \sum_1^N e_i \otimes \delta _i$ (recall \eqref{eq:656}) so that $E_+^0$ is {an isometry}, and so $\norm{E_+^0 u } = \norm{u} =1$.
Because $s_{2,j}$ are decreasing in $j$, we get that $C_S \le s_{2,j} \le \max(\norm{S_1},\norm{S_2})$ for $j = 1,\dots, N$.
A similar argument can be used to establish the same bound for each $s_{1,j}$.
}

Because the law of Gaussian ensembles is invariant under unitary conjugations (see for instance \cite[\S 13]{Hager2008}), we are free to choose any two orthonormal bases of $\mathcal{H}$ when defining $Q_\omega$ (in \eqref{random_pertubation}). 
{We will choose these bases in the following way: }let $f^1_j$ and $f^2_j$ be orthonormal bases of $\C^N$ which are eigenfunctions of $\sqrt{(E_-^0 S_1)(E_-^0 S_1)^*}$ and $\sqrt{(S_2 E_+^0 )^* (S_2 E_+^0)}$, respectively, corresponding to eigenvalues $s_{1,j}$ and $s_{2,j}$.
Then, let $e^1_j$ and $e^2_j$ (for $j \in \N$) be orthonormal bases of $\mathcal{H}$ such that
\begin{align}
e^1_j = \frac{1}{s_{1,j}} (E_-^0 S_1)^* f^1_j &&\text{and} && e^2_j = \frac{1}{s_{2,j}} (S_2 E_+^0) f^2_j
\end{align}
for $j = 1, \dots, N$. 
By this construction, we see that
\begin{align}
E_-^0 Q_\omega E_+^0 = E_-^0 S_1 \left( \sum_{j,k=1}^\infty \alpha_{j,k} e^1_j \otimes e^2_k \right) S_2 E_+^0
= \sum_{1 \le j, k \le N} s_{1,j} s_{2,k} \alpha_{j,k} f_j^1 \otimes f_k^2. \qquad \qedhere
\end{align}
\end{proof}

Let $T_1 = \diag(\set{s_{1,j}}_{1\le j\le N})$ and $T_2 = \diag (\set{s_{2,j}}_{1 \le j\le N})$ where $s_{1,j}$ and $s_{2,j}$ are defined in Proposition \ref{claim.1}.
With this, we can rewrite the second term in \eqref{first Epm} (with respect to the bases $f^1_j$ and $f^2_k$) as
\begin{align}
\delta \hat Q_{\omega} =\delta T_1 \circ ( (\alpha _{j,k})_{1\le j,k\le N} + \tilde T ) \circ T_2 \label{eq:317}
\end{align}
where
\begin{align}
\tilde T = T_1^{-1 } \circ T \circ T_2^{-1}.\label{eq:765}
\end{align}
Here, we recall $T$ was defined in \eqref{eq:313} as the lower order terms in the expansion of $E_{-+}^\delta$.

Let $\hat \mu$ be the law of $(\alpha_{j,k})_{1\le j,k\le N} + \tilde T$. The following lemma, proven in \cite[Eq. 9.18]{Hager2008}, bounds the measure $\hat \mu$.

{Recall from} Assumption \ref{condition:5} that $S_1$ is a trace-class operator{, so} it can be written as a product of Hilbert--Schmidt operators $S_{1,1}$ and $S_{1,2}$ (see for instance \cite[Chapter 10]{retherford1993hilbert}).
Moreover, we can choose these Hilbert--Schmidt operators\footnote{{Indeed, by polar decomposition} $S_1 = \sum_{j}\lambda_j e_j \otimes f_j$ for $e_j$ and $f_j$ orthonormal bases of $\mathcal{H}$ and $\lambda_j$ the singular values of $S_1$, then we can define $S_{1,1} \coloneq \sum_j \sqrt{\lambda_j} e_j \otimes f_j$ and $S_{1,2} \coloneq \sum_j \sqrt{\lambda_j} f_j \otimes f_j$.} such that
\begin{align}
    \norm{S_{1,1}}_{\operatorname{HS}} = \norm{S_{1,2}}_{\operatorname{HS}} = \sqrt{\norm{S_1}_{\operatorname{Tr}}}
\end{align}
and $\norm{S_{1,1}} = \norm{S_{1,2}} = \sqrt{\norm{S_1}}$.

\begin{lemma}\label{lemma.1}
For $M > 0$, define
\begin{align}
    \mathcal{Q}_M \coloneq \set{(\alpha_{j,k} )_{j,k\in \N} : \norm{ S_{1,2} \circ\left(\sum_{j,k = 1}^{\infty} \alpha_{j,k}e_j^1 \otimes e_k^2\right) \circ S_2}_{\operatorname{HS}} < M}, \label{eq:QM}
\end{align} 
let $\mu_N$ denote the probability measure on complex $N\times N$ matrices given by:
\begin{align}
    \mu_{N} = \prod_{j,k=1}^N \left (  e^{-\abs{\alpha_{j,k}}^2 } \frac{\dd m(\alpha_{j,k})}{\pi}   \right )
\end{align}
{where $m$ is the Lebesgue measure on $\C$.}

{
Then for $(\alpha)_{j,k\in \N}$ in $\mathcal{Q}_M$, the law of $(\alpha_{j,k})_{1\le j,k\le N} + \tilde T$ (as in \eqref{eq:765}), denoted by $\hat \mu$, satisfies the following bound:
\begin{align}
\hat \mu \le \left  (1 + \mathcal{O}(1) \frac{\delta M^3}{ \sqrt{\alpha}} \right ) \mu_N. 
\end{align}
}
{Concretely, if $A \subset \mathcal{Q}_M$ is measurable, then $\hat \mu (A) \le (1 + \mathcal{O}(1 ) \delta M^3\alpha ^{-1/2}) \mu _N (A)$.}
\end{lemma}

\noindent \textbf{Step 4. Estimate the probability $E_{-+}^\delta$ is invertible}

Observe by \eqref{first Epm}
\begin{align}
    \det(E_{-+}^\delta)  = \delta ^N \det (\delta^{-1} E_{-+}^0 + \hat Q_\omega).\label{eq.26}
\end{align}

Therefore, on $\mathcal Q_M$, the event that the right-hand side of \eqref{eq.26} is zero will have probability zero. 
This is because the entries of the random matrix have probability densities that are absolutely continuous with respect to the Lebesgue measure (using Lemma \ref{lemma.1}), and the zero \edit{set} of the characteristic polynomial has codimension $1$.

Therefore, $\Spec(D \cap A^\delta)$ is discrete as long as we can build a Grushin problem (which requires a norm bound on $Q_\omega$, \eqref{eq:condition on Q_omega}) and $(\alpha_{j,k})_{j,k\in \N}$ in the random perturbation belongs to $\mathcal Q _M$. 
Concretely,
\begin{align}
   \P( \Spec(D \cap A^\delta)  \text{ is discrete})\ge \P(\set{\norm{Q_\omega} < \delta^{-1} \alpha ^{1/2} } \cap \mathcal{Q}_M ).\label{eq329}
\end{align}
If $M  < \delta^{-1}\alpha^{1/2}\norm{S_{1,1}}^{-1}$, then
\begin{align}
    \norm{Q_\omega}_{\operatorname{HS}}\le \norm{S_{1,1}} \norm{S_{1,2} \circ \left( \sum_{\alpha _{j,k}} e_{j}^1 \otimes e_k^2 \right)  \circ S_2}_{\operatorname{HS}}< \delta ^{-1}\alpha^{1/2}\label{eqe:330}
\end{align}
so that $ \mathcal{Q}_M\subset \set{\norm{Q_\omega} < \delta^{-1} \alpha ^{1/2} }$.
Therefore \eqref{eq329} becomes
\begin{align}
    \P( \Spec(D \cap A^\delta)  \text{ is discrete})&\ge \P(\mathcal{Q}_M)\\
    &> 1 - C_0 \exp\left( \frac{C_0 \norm{S_{1,2}}_{\operatorname{HS}}^2\norm{S_2}_{\operatorname{HS}}^2 - M^2}{2\norm{S_{1,2}} \norm{S_2}}\right).\label{eq:3.22}
\end{align}
{Here the second inequality follows from Lemma \ref{lemma.norm of qomega}.}
We can select $\e \in (0,1)$ and set $M= \e\delta^{-1} \alpha^{1/2} \norm{S_{1,1}}^{-1}$. 
In this case, we have:
\begin{align}
     \frac{C_0 \norm{S_{1,2}}_{\operatorname{HS}}^2\norm{S_2}_{\operatorname{HS}}^2 - M^2}{2\norm{S_{1,2}} \norm{S_2}}&=  \frac{C_0\norm{S_{1,1}} ^2\norm{S_{1,2}}_{\operatorname{HS}}^2\norm{S_2}_{\operatorname{HS}}^2 - \e^2 \alpha \delta^{-2} }{2\norm{S_{1,1}}^2 \norm{S_{1,2}} \norm{S_2}}\\
     &= \frac{C_0\norm{S_{1}} \norm{S_{1,2}}_{\operatorname{HS}}^2\norm{S_2}_{\operatorname{HS}}^2 - \e^2 \alpha \delta^{-2}}{2\norm{S_1}^{3/2} \norm{S_2}}\\
      &= \frac{C_0\norm{S_1}\norm{S_1}_{\operatorname{Tr}}\norm{S_2}_{\operatorname{HS}}^2 - \e^2 \alpha \delta^{-2}}{2\norm{S_1}^{3/2} \norm{S_2}}
\end{align}
where in the second equality, we use that $\norm{S_{1,1}} =\norm{S_{1,2}} = \norm{S_1}^{1/2}$, and in the third equality, we use that $\norm{S_{1,2}}_{\operatorname{HS}}^2 = \norm{S_1}_{\operatorname{Tr}}$.
We can then take the limit as $\e \to 1$, so that \eqref{eq:3.22} becomes
\begin{align}
    \P( \Spec(D \cap A^\delta)  \text{ is discrete})&\ge 1 - C_0 \exp\left (\frac{C_0\norm{S_1}\norm{S_1}_{\operatorname{Tr}}\norm{S_2}_{\operatorname{HS}}^2 - \alpha \delta^{-2}}{2\norm{S_1}^{3/2} \norm{S_2}}\right) .
\end{align}

\noindent \textbf{Step 5. Consider the case where $N(\alpha) = 0$ and extend the result to all $D_j$.}

{
We recall, when building the Grushin problem, we assumed that $N  >0$ (where $N$ was defined in \eqref{eq:694} as the number of small singular values depending on $\alpha$).
If $N = 0$, then this implies that $0$ is not a singular value, so $A- z_0$ is invertible.
In this case, we can build the Grushin problem by setting $\mathcal{P} = A-z_0$ and $\mathcal E = \sum_1^\infty t_i^{-1} e_i \otimes f_i$.
This allows us to build the inverse for $A+\delta Q_\omega - z_0$ as $\mathcal{E}(1 + \delta Q_\omega \mathcal E)^{-1}$ with a Neumann series, provided that $\delta \alpha^{-1/2} \norm{Q}_\omega < 1$.
We then get $\P( \Spec(D \cap A^\delta)  \text{ is discrete})\ge \P(\set{\norm{Q_\omega} < \delta^{-1} \alpha ^{1/2} } )$, which is bounded below by \eqref{eq:3.22}, and we get the same result.
}

For each $j\in \mathcal{J}$, let $\mathcal{E}_j$ be the event $\Spec(A^\delta ) \cap D_j$ has discrete spectrum. 
But by the above discussion, if $M < \delta ^{-1} \alpha^{1/2} \norm{S_{1,1}}^{-1}$, then $\mathcal{E}_j \supset \mathcal Q_M$.
Therefore:
\begin{align}
    \P \left(\bigcap_{j \in \mathcal{ J}} \mathcal E_j\right) \ge \P ( \mathcal Q_M)
\end{align}
which is estimated in \eqref{eq:3.22}.\end{proof}

We now prove Theorem \ref{theorem:quantitative}, the more quantitative version of Theorem \ref{thm:general result}, by estimating the size of the smallest singular value.

\begin{proof}[Proof of Theorem \ref{theorem:quantitative}]
\

\noindent \textbf{Step 1. Relate singular value to $E_{-+}^\delta$}

We begin by relating the smallest singular value of $A^\delta - z_0$ to the absolute value of the determinant of $ E_{-+}^\delta$ (as long as a Grushin problem can be constructed).
We define $0\le t_{1,\delta} \le t_{2,\delta}\le \cdots$ as the increasing sequence of singular values of $A^\delta - z_0$.
We can write
\begin{align}
    \P(t_{1,\delta }\le a ) \le \P( \set{t_{1,\delta} \le a } \cap \mathcal{Q}_M) + \P(\set{t_{1,\delta} \le a } \cap \mathcal{Q}_M^c)\label{eq:793}
\end{align}
recalling the definition of $\mathcal{Q}_M$ given in \eqref{eq:QM}.
The second term of the right-hand side of \eqref{eq:793} is bounded using Lemma \ref{lemma.norm of qomega}:
\begin{align}
   \P(\set{t_{1,\delta} \le a } \cap \mathcal{Q}_M^c)\le  \P(\mathcal{Q}_M^c) \le  C_0 \exp\left( \frac{C_0 \norm{S_{1,2}}_{\operatorname{HS}}^2\norm{S_2}_{\operatorname{HS}}^2 - M^2}{2\norm{S_{1,2}} \norm{S_2}}\right).\label{eq:819}
\end{align}

As long as $M < \delta ^{-1}\alpha ^{1/2} \norm{S_{1,1}}^{-1}$, we can set up the same Grushin problem as in the proof of Theorem \ref{thm:general result} (recalling \eqref{eq:condition on Q_omega}).
We now use a general fact about Grushin problems.
Applying the Schur complement formula and properties of singular values (see for instance \cite[Lemma 18]{Vogel2020}), we get that
\begin{align}
\frac{t_1(E_{-+}^\delta )}{t_1(E_{-+}^\delta )\norm{E^\delta} + \norm{E_-^\delta}\norm{E_+^\delta}} \le     t_{1,\delta}\label{eq:826}
\end{align}
where $0\le t_1(E_{-+}^\delta) \le t_2(E_{-+}^\delta ) \le \cdots \le t_{N}(E_{-+}^\delta)$ are the singular values of $E_{-+}^\delta$.
By expanding the Neumann series expansion for the Grushin problem (similarly to \eqref{eq:607}), we see that, in the event $\mathcal{Q}_M$, we have the following bounds:
\begin{align}
t_1(E_{-+}^\delta) =  \mathcal O (1),  && \norm{E^\delta}   = \mathcal O (\alpha ^{-1/2}),\\
\norm{E_-^\delta} = \mathcal O (1), && \norm{E_+^\delta} = \mathcal O (1),
\end{align}
which combined with \eqref{eq:826}, gives us
\begin{align}
\alpha ^{1/2} t_1 (E_{-+}^\delta ) \le  C  t_{1,\delta}\label{eq:834}
\end{align}
for some constant $ C >0 $ where we use that $\alpha < \alpha_{\rm{max}}$.

Next, observe that
\begin{align}
    \abs{\det E_{-+}^\delta } = \prod_{j=1}^N t_j (E_{-+}^\delta )\le t_1(E_{-+}^\delta) (t_N(E_{-+}^\delta ))^{N-1}  = t_1(E_{-+}^\delta) \norm{E_{-+}^\delta}^{N-1}.\label{eq:840}
\end{align}
Within the event $\mathcal{Q}_M$, with $M < \delta ^{-1} \alpha^{1/2}\norm{S_{1,1}}^{-1}$, we can use \eqref{eq:607} to get that $\norm{E_{-+}^\delta} = \mathcal O (\alpha^{1/2})$.
Using this, \eqref{eq:840} can be rearranged as
\begin{align}
    t_1(E_{-+}^\delta ) \ge  \frac{\abs{\det (E_{-+}^\delta )}}{(C \sqrt{\alpha})^{N-1}}\label{eq:845}
\end{align}
for some (possibly new) constant $C>0$.
Therefore, combining \eqref{eq:845} with \eqref{eq:834}, we get that
\begin{align}
    \set{(t_{1,\delta} \le a)  \cap \mathcal{Q}_M }\subset \set{\left(\abs{\det(E_{-+}^\delta)} < a C^{N-1} \sqrt{\alpha}^{N-2}\right) \cap \mathcal{Q}_M} \label{eq:981}
\end{align}
for some (possibly new) constant $C> 0$.

\noindent \textbf{Step 2. Estimate the probability that $|\det(E_{-+}^\delta)|$ is small}

Recall, by \eqref{eq:317}
\begin{align}
E_{-+}^\delta = E_{-+} + \delta T_1 \circ (  (\alpha_{j,k})_{1\le j,k\le N} + \tilde T ) \circ T_2
\end{align}
where the definition of $T$ is given in \eqref{eq:313}, $T_{1,2}$ are defined in the discussion preceding \eqref{eq:317}, and $\tilde T \coloneq T_1^{-1} \circ T \circ T_2^{-1}$.
This gives us
\begin{align}
&\set{|\det E_{-+}^\delta | >a C^{N-1} \sqrt{\alpha}^{N-2}} \\
&\qquad =  \set{ |\det  \left ( (\alpha_{j,k})_{1\le j,k\le N} + T + \delta ^{-1 }T_1^{-1} E_{-+}^0T_2^{-1} \right ) | >a C^{N-1} \sqrt{\alpha}^{N-2}\delta ^{-N}\det (T_1 T_2)^{-1}}.
\end{align}
We also recall that Lemma \ref{lemma.1} estimates the law of $(\alpha_{j,k})_{1\le j,k\le N} + T$ by the law of an $N\times N$ Gaussian ensemble.

We next use the following Lemma from \cite[Proposition 7.3]{Hager2008} estimating the size of the determinant of a Gaussian ensemble added to a deterministic matrix.

\begin{lemma}\label{lemma.2'}
If $V_\omega = (\alpha _{i,j} )_{1\le i,j\le N}$ with $\alpha_{i,j} \sim \mathcal{N}_\C(0,1)$ i.i.d., and $D\in \C^{N\times N}$, then there exist $C_1, C_2 > 0$ (independent of $D$) such that for $c>0$:
\begin{align}
\P(  \abs{\det (D+V_\omega) }\le c ) \edit{\; \le \;}  C_1 c \exp \left (- \frac{1}{2} (C_2 + (N-\frac{1}{2} )\log N - 2N )   \right ) \label{eq.3'}
\end{align}
as long as
\begin{align}
c \le \exp\left(\frac{C_2 + (N + 1/2)\log(N) - 2N}{2}\right)\label{eq:889}
\end{align}
for $C_2$ the same constant as in \eqref{eq.3'}.
\end{lemma}

We now apply Lemmas \ref{lemma.1} and \ref{lemma.2'} (setting $D = \delta ^{-1} T_1^{-1} E_{-+}^0T_2^{-1}$ and assuming \eqref{eq:889} holds for now) to get that
\begin{align}
&\P ( \set{ \abs{\det E _{-+}^\delta }\le  a C^{N-1} \sqrt{\alpha}^{N-2}} \cap \mathcal{Q}_M) \\
&\qquad \le \left ( 1 + \mathcal{O}(1)  \frac{\delta M^3}{\sqrt{\alpha}} \right )  \P \left (  | \det  (V_\omega  +D) | \le a C^{N-1} \sqrt{\alpha}^{N-2}\delta^{-N} \det (T_1 T_2)^{-1} \right ) \\
&\qquad \le  \left ( 1 + \mathcal{O}(1)  \frac{\delta M^3}{\sqrt{\alpha}} \right )  \left(a C^{N-1} \sqrt{\alpha}^{N-2}\delta ^{-N}\det(T_1 T_2 )^{-1} \right) \\
& \qquad \qquad \qquad \qquad \qquad \cdot\exp \left  ( -\frac{1}{2} (C_2 + ( N - \frac{1}{2} ) \log (N) - 2N ) \right ).\label{eq228}
\end{align}
Letting $M < \delta ^{-1} \alpha ^{1/2} \norm{S_{1,1}}^{-1}$, we get that
\begin{align}
    \frac{\delta M^3}{\sqrt{\alpha}} < \frac{\alpha^{2} \delta ^2}{\norm{S_{1,1} }^3} = \frac{\alpha^{2} \delta ^2}{\norm{S_{1} }^{3/2}}
\end{align}
where we recall that $\norm{S_{1,1}} = \norm{S_1}^{1/2}$.
By Proposition \ref{claim.1}, the determinants of $T_1$ and $T_2$ are bounded below, so that
\begin{align}
    \det(T_1 T_2 )^{-1} \le C_S^{-2N}.\label{eq:906}
\end{align}
There exist positive constants $C_3$ and $C_4$ such that
\begin{align}
    \frac{-1}{2} ((N-1/2) \log (N) - 2N) \le -C _3 N \log (N) + C_4. \label{eq:910}
\end{align}

\noindent \textbf{Step 3. Estimate the probability the smallest singular value is small}

We therefore get (by \eqref{eq:981}):
\begin{align}
\P(\set{t_{1,\delta} \le a} \cap \mathcal{Q}_M)\le C \frac{\alpha^{2+(N-1)/2} }{\delta^{N-2}C_S^{2N}\norm{S_{1}}^{3/2}}N^{-C_3N}a. \label{eq:914}
\end{align}

For \eqref{eq:889} to hold, we require that
\begin{align}
    aC^{N-1} \sqrt{\alpha}^{N-2}\delta^{-N}\det(T_1T_2)^{-1} \le \exp \left ( \frac{C_2 + (N+1/2) \log (N) - 2N}{2} \right). \label{eq:916}
\end{align}
Using \eqref{eq:906} and a similar bound as in \eqref{eq:910}, \eqref{eq:916} holds if
\begin{align}
    a\le \exp\left( C_5 N \log (N)  + N \log(\delta) -
\left(\frac{N-2}{2}\right) \log(\alpha) \right)
\end{align}
for some $C_5 >0$. 
Combining \eqref{eq:914} with \eqref{eq:819} (setting $M = \e \delta ^{-1} \sqrt \alpha \norm{S_{1}}^{-1/2}$) and taking the limit as $\e \to 1$ gives us the Theorem.
\end{proof}

\section{Applications to twisted bilayer graphene}
\label{sec:TBG}
In this section we build Hilbert--Schmidt operators to satisfy Hypothesis \ref{def1} for the operator
\begin{align}
D_h(\beta) =  \mat{2hD_{\bar z } & U(z) \\ U(-z) & 2h D_{\bar z}} \label{eq:Dh defined}
\end{align}
first defined in \eqref{eq:first define of Dh} recalling that $2\dbar _z \coloneq (\p_{\Re z} + i \p _{\Im z })$. 
Here we absorb $\beta$ into the definition of $U$, so that $U(z) \coloneqq  \beta \sum_{k=0}^2 \omega^k \exp((z\bar \omega^k - \bar z \omega^k)  /2 )$.

{
By \cite[Proposition 2.3]{Becker2020}, $D_h(\beta)$ is an unbounded Fredholm operator on the domain $H^1(\C / \Gamma; \C^2)$ with $\rho_F^{(0)}(D_h(\beta)) = \C$ (recall $\rho_F^{(0)}$ is defined in equation \eqref{eq:rhoFA0}).

We will ultimately apply Theorem \ref{thm:general result} to $D_h(\beta)$ which requires showing that the spectrum of $D_h(\beta)^* D_h(\beta)$ on its natural domain is discrete.
This follows by the Sobolev embedding.
In fact, $D_h (\beta)^* D_h (\beta)$ is a self-adjoint operator on $H^2 (\C / \Gamma ; \C^2)$, so that $(D_h (\beta)^* D_h (\beta) - i)^{-1}$ is a bounded operator from $L^2 (\C / \Gamma)$ to $H^2 (\C /\Gamma)$.
According to the Rellich--Kondrachov Theorem, $H^2 (\C /\Gamma)$ is compactly embedded in $L^2 ( \C / \Gamma)$ so that $$(D_h (\beta)^* D_h (\beta) - i)^{-1} \colon L^2 (\C / \Gamma) \to L^2 (\C /\Gamma)$$ is a compact operator. 
Therefore $(D_h (\beta)^* D_h (\beta) - i)^{-1}$ has discrete spectrum, which implies that 
$D_h (\beta)^* D_h (\beta)$ has discrete spectrum.

We want to construct operators $S_1$ and $S_2$ that satisfy Hypothesis \ref{def1} for the random perturbation $Q_\omega$ defined in \eqref{random_pertubation}.
}

\begin{claim}\label{claim:existenceofrandomtbg}
There exist constants $C_1,C_2 > 0$ such that if $\chi(z,\zeta) \in C_0^\infty (T^*(\C /\Gamma) ; [0,1])$ is identically $1$ when $|\zeta|^2 < C_1 C_2$, then
\begin{align}
S_1 = S_2 = \Op_h\mat{\chi & 0 \\ 0 & \chi} \label{eq:S1 S2 defined}
\end{align}
satisfy Hypothesis \ref{def1} for $\alpha < C_2$, where $\Op_h$ is the quantization of functions on $T^*(\C/\Gamma)$ defined in Definition \ref{def:quantization_thing}.
\end{claim}

\begin{proof}
{
To show that $S_1$ and $S_2$ satisfy Hypothesis \ref{def1} for an $\alpha$, we must show there exists a $C_S  > 0$ such that  $\norm{S_1 u} , \norm{S_2 u} \ge \norm{u}$ for every $u\in H^2 (\C /\Gamma ; \C^2)$ satisfying $D_h(\beta)^*D_h(\beta) u = t_i^2 u $ for $t_i^2 < \alpha$. 
Let us now fix such a $t_i$ and $u$.

To establish the lower bound, we will use the symbolic calculus of matrix-valued symbols (which is reviewed in the Appendix).
}

{To apply this calculus, we use the following notation.} We write elements of $T^*(\C / \Gamma)$ as $(x + i y,\xi + i \eta)$ where $\xi,\eta,x,y\in \R$ are such that $x+iy \in \C / \Gamma$.
We will also write $\zeta = \xi + i \eta$ and $z = x + i y$.

The operator $D_h(\beta)$ (defined in \eqref{eq:Dh defined}) has principal symbol
\begin{align}
\mat{ \zeta & U(z) \\ U(-z) & \zeta }.
\end{align}

Define $\SAAA \coloneqq (D_h(\beta))^* (D_h(\beta))$ {with domain $H^2 (\C /\Gamma ; \C^2)$}. 
{From this, we see that $\SAAA  u = t_i^2 u$} and
\begin{align}
q_0 \coloneqq \sigma_0 (\SAAA)=  \mat{|\zeta|^2 + |U(-z)|^2 & \bar \zeta U(z) + \zeta \overline{U(-z)}  \\ \overline{U(z)} \zeta + \bar \zeta U(-z)  & |\zeta|^2 + |U(z)|^2 }
\end{align}
where $\sigma_0(P)$ denotes the principal symbol of an operator $P$ (see Definition \ref{def:principal symbol}).    
We would like to build a $\chi \in C_0^\infty (T^*M ; [0,1])$ such that a parametrix can be constructed for $\SAAA - t_i^2 + i \Op_h(\chi I_2)$, where
\begin{align}
I_2 \coloneqq \mat{1 & 0 \\ 0 & 1}.
\end{align}
{Such a parametrix will then be used to show that $S_1 u = S_2 u= u + \mathcal{O}(h^\infty)$, which will provide the desired lower bound of $\norm{S_1 u}$ and $\norm{S_2 u}$.}
{To construct such a parametrix}, we must show that the determinant of the principal symbol of $\SAAA-t_i^2 + i \Op_h(\chi I_2)$ is uniformly bounded from below.

Let $\lambda _{\pm} = \lambda_\pm (t_i)$ denote the eigenvalues of $q_0 - t_i^2 {I_2}$. Fix $\e \in (0,1)$, {$C_2 >0$ sufficiently large (which we will specify later)}, and define
\begin{align}
    K \coloneqq \set{  (z,\zeta) \in T^* (\C / \Gamma) : |\zeta|^2 \le (1 - \e)^{-1}C_2 }, \label{eq:1066}
\end{align}
and let $\tilde \chi \in C_0^\infty (T^*(\C / \Gamma); [0,1])$ be identically $1$ on $K$ (which is a compact set) and zero on the set
\begin{align}
    \set{  (z,\zeta) \in T^* (\C / \Gamma) : |\zeta|^2 > C_1 C_2 }
\end{align}
where $C_1 = c(1 - \e)^{-1}$ and $c  > 1$.
We can then compute that
\begin{align}
    |\det (\sigma_0 ( \SAAA - t_i^2 + i \Op_h(\chi I_2))|^2 &= |\lambda_+ \lambda_- + i \chi (\lambda_+  + \lambda_- ) - \chi^2 |^2 \\
    &= (\lambda_+ \lambda_-)^2 + \chi^4 + \chi^2 (\lambda _+^2 + \lambda_-^2 ),\label{eq:1071}
\end{align}
which is bounded below by $\chi^4$ so that for $(z,\zeta) \in K$, \eqref{eq:1071} is bounded below by $1$.

We now aim to provide a lower bound of $\lambda_+ \lambda_- $ for $(z,\zeta) \in K^c$. Note that:
\begin{align}
    \lambda_+ \lambda_- = \det (q_0 - t_i^2 ) = f_1 (z,\zeta) f_2 (z,\zeta) - f_3(z,\zeta) \label{eq49}
\end{align}
where
\begin{align}
    f_1 (z,\zeta) &\coloneqq |\zeta|^2 + |U(-z)|^2 - t_i^2, \\
    f_2 (z,\zeta) & \coloneqq |\zeta|^2 + |U(z)|^2 - t_i^2, \\
    f_3 (z,\zeta) & \coloneqq |\bar \zeta U(z) + \zeta \overline{U(-z)}|^2.
\end{align}
Recalling that $|U(z)|$ is bounded, we see that $|f_3| \le  C |\zeta|^2$ where $C$ depends on the maximum of $|U(z)|$.
Furthermore, for $j=1,2$
\begin{align}
    f_j(z,\zeta) {\ge |\zeta|^2 - t_i^2\ge |\zeta|^2 -\alpha \ge |\zeta|^2 - C_2} \ge \e  |\zeta|^2 \label{eq413} 
\end{align}
using \eqref{eq:1066}.
Therefore, using \eqref{eq:1066}, \eqref{eq49} and \eqref{eq413}, we get that
\begin{align}
    \lambda_+\lambda_-  \ge \e^2 |\zeta|^4 - C |\zeta|^2 =|\zeta|^2 (\e^2 |\zeta|^2 - C) \ge  (1-\e)^{-1} C_2\left( \e^2(1-\e)C_2 -C\right)  > 0
\end{align}
as long as
\begin{align}
    C_2^2 > \frac{C}{\e^2(1-\e)}. \label{eq:1138}
\end{align}
{For a sharper lower bound on $C_2$, we may minimize the right-hand side of \eqref{eq:1138} by letting $\e = 2/3$, in which case we have $(\e^2 (1-\e))^{-1} = 27/4$.
In this case we require $C_2 > \sqrt{27 C} / 2$, and $C_1$ in the statement of the Claim is $3$.
}

{Having shown a lower bound on the determinant of the principal symbol of $\SAAA - t_i^2 + i \Op_h(\tilde \chi I_2)$}, we can invert $\SAAA - t_i^2 + i \Op_h(\tilde \chi I_2)$ using Proposition \ref{claim.parametrix}.
{Let $P$ be a parametrix} of $\SAAA - t_i^2 + i\Op_h(\tilde \chi I_2) $.
Now let $ \chi \in C_0^\infty (T^* (\C / \Gamma) ; [0,1])$ be identically $1$ on the support of $\tilde \chi$.
Recalling that $\SAAA u = t_i^2 u$, we get that
\begin{align}
\Op_h(I_2 -  \chi I_2) u &= \Op_h(I_2 -  \chi I_2) P(\SAAA - t_i^2 + i\Op_h(\tilde \chi I_2)) u + \mathcal{O}_{L^2 \to L^2} (h^\infty) \\
&= \Op_h  (I_2 -  \chi I_2) P(i \Op_h(\tilde \chi I_2) ) u + \mathcal{O} (h^\infty)= \mathcal{O}_{L^2\to L^2}(h^\infty). 
\end{align}
Here we use that the composition of symbols with disjoint support yields $\mathcal{O}_{L^2 \to L^2}(h^\infty)$ errors.
We then see that
\begin{align}
\Op_h( \chi I_2)  u = u + \mathcal{O} (h^\infty)
\end{align}
so we get that:
\begin{align}
\norm{  \Op_h( \chi I_2) u } \ge (1 + \mathcal{O} (h^\infty)  ) \norm{u}
\end{align}
so that $\Op_h( \chi I_2)$ satisfies Hypothesis \ref{def1}.
\end{proof}

\begin{theorem}[Perturbation of $D_h(\beta)$]\label{theorem_pert_of_DhB}
Suppose $D_h(\beta)$ is defined by \eqref{eq:first define of Dh}, $S_1$ and $S_2$ are defined by \eqref{eq:S1 S2 defined}, $0 < \delta < h^{\kappa}$ with $\kappa >  2$, and $Q_\omega$ is defined by \eqref{random_pertubation}. 
Then $D_h(\beta) + \delta Q_\omega$ has discrete spectrum with probability at least
\begin{align}
1 - C_1 \exp(-C_2/h^{2\kappa})
\end{align}
for positive constants $C_1$ and $C_2$ (i.e. with overwhelming probability).
\end{theorem}
\begin{proof}

In Claim \ref{claim:existenceofrandomtbg}, we constructed $S_1$ and $S_2$ satisfying the hypotheses of Theorem \ref{thm:general result}.
Moreover, because the symbols of $S_1$ and $S_2$ have compact support, they are both trace-class.

We therefore can apply Theorem \ref{thm:general result} to get that there exists a $C_0  > 0 $ such that $D(h) + \delta Q_\omega$ has discrete spectrum with probability at least:
\begin{align}
    \max\left(1 - C_0 \exp \left ( \frac{C_0 \norm{S_1} \norm{S_1}_{\operatorname{Tr}} \norm{S_2}_{\operatorname{HS}}^2 - \alpha  \delta ^{-2}}{ 2 \norm{S_1}^{3/2} \norm{S_2}} \right ) , 0 \right).
\end{align}
By construction, $\norm{S_1}_{\operatorname{Tr}} , \norm{S_2}^2_{\operatorname{HS}} = \mathcal{O} (h^{-2}) $, and $\norm{S_1}, \norm{S_2} \sim 1$. 
Therefore, if $0 < \delta < h^{\kappa}$ with $\kappa >  2$, we get that $D(h) + \delta Q_\omega$ has discrete spectrum with probability at least
\begin{align}
1 - C_1 \exp\left( -C_2/h^{2\kappa}\right) 
\end{align}
for positive constants $C_1$ and $C_2$.
\end{proof}

\appendix

\section{Quantization of matrix valued functions on \texorpdfstring{$\C / \Gamma$}{C/Γ}}\label{section3}

In this appendix, we provide the background and pseudo-differential calculus to construct operators $S_1$ and $S_2$ for our application to twisted bilayer graphene and, in particular, to $D_h(\beta)$ as defined in \eqref{eq:first define of Dh}.

{
The ultimate goal is to prove a composition result for certain quantized operators (Proposition \ref{claim:composition}) and a parametrix construction (Proposition \ref{claim.parametrix}). 
We will prove this by using well-established results on the Weyl quantization of scalar-valued functions on $\R^d$.

\edit{The results of this section follow via mild modifications of standard results for pseudo-differential operators with scalar-valued symbols.  
The main difference in the matrix-valued symbol case, is that when constructing a parametrix, the inverse of the symbol must be used.
A careful verification of the usual parametrix construction (which relies on a well-defined symbol class with a calculus) is required.
}

A summary of this appendix is as follows.
\begin{enumerate}
    \item Define the class of functions ($S^k(T^*(\C / \Gamma); \C^2)$) we want to quantize (Definition \ref{def:1275}).
    \item Show that the usual $h$-Weyl quantization of these functions induces bounded linear maps between appropriate Sobolev spaces on $\C / \Gamma$ (Proposition \ref{prop:1325}).
    \item Define the quantization of $A \in S^k(T^*(\C / \Gamma); \C^2)$ as the restriction of the $h$-Weyl quantization to a Sobolev space on $\C / \Gamma$ (Definition \ref{def:quantization_thing}).
    \item Use standard results about the $h$-Weyl quantization of scalar-valued functions on $\R^d$ to get a composition and parametrix result (Propositions \ref{claim:composition} and \ref{claim.parametrix}). 
\end{enumerate}
}

We aim to quantize matrix-valued symbols on the cotangent space of $\C / \Gamma$. 
We recall that the lattice $\Gamma$ is given by $ 4\pi (i\omega \Z \oplus i \omega ^2 \Z)$ where $\omega = e^{2\pi i /3}$.
The discussion in this section will work for any lattice of $\R^d$ spanned by $d$ linearly independent vectors, however, we only consider $\Gamma \subset \mathbb C$ to keep the discussion as explicit as possible.

{We first recall the $h$-Weyl quantization of symbols on $\R^{2d}$.
For
\begin{align}\label{eq:1317}
    a (x,\xi) \in S^0(\R^{2d}) \coloneq \set{a(x,\xi) \in C^\infty (\R^{2d}) : \forall  \alpha \in \N^{2d} \  \exists \  C_\alpha >0 \  \text{s.t.} \  |\p^\alpha_{x,\xi} a(x,\xi) | \le C_\alpha },
\end{align}
we define $\Op_h^w(a)$ acting on $u\in \mathcal{S}(\R^d)$ (Schwartz) by
\begin{align}
    \Op_h^w(a) u (x) \coloneq \frac{1}{2\pi h} \iint_{\R^{2d}} e^{\frac{i}{h}(x-y)\cdot\xi}a\left(\frac{x+y}{2},\xi\right) u(y) \dd y \dd \xi \in \mathcal{S}(\R^d).
\end{align}
By duality, $\Op_h^w(a)$ extends to an operator from $\mathcal{S}'(\R^d)$ to itself.
If $a$ has certain decay properties, $\Op_h^w(a)$ can be shown to map boundedly between Sobolev spaces.

Identifying $\C$ with $\R^{2d}$, this allows us to quantize functions on $\C$.
We can also consider matrix-valued symbols by quantizing each component individually.
What requires some discussion is how to construct a quantization of symbols on $\C / \Gamma$ that map from some set of functions on $\C / \Gamma$ to some other set of functions on $\C / \Gamma$.

Let us begin by defining the space of symbols we wish to quantize.} 
These symbols will be functions on the cotangent bundle $T^*(\C /\Gamma)$. 
 The cotangent bundle $T^*(\C / \Gamma)$ can be identified with $(\C / \Gamma ) \times \R^2$.
 Elements of $T^*(\C / \Gamma)$ can be written as $(z,(\xi,\eta))$ with this identification, where $z\in \C / \Gamma$ and $(\xi,\eta) \in \R^2$.
 We also identify $\R^2$ with $\C$ by writing $\zeta = \xi + i \eta$, so that elements of $T^*(\C / \Gamma)$ can also be written $(z,\zeta)$.

\begin{defi}[$S^k(T^*(\C / \Gamma))$]
For $k\in \R$, we define $S^k(T^*(\C/\Gamma))$ as the set of functions $f$ in $C^\infty (T^*(\C/\Gamma))$ such that for all $\alpha,\beta \in \N^{2}$, there exists $C_{\alpha ,\beta} > 0 $ such that:
\begin{align}
|\p^\alpha_{z,\bar z }  \p^\beta_{\zeta,\bar \zeta} f(z,\zeta)| \le C_{\alpha,\beta} (1 + |\zeta|)^{k - |\beta|}.
\end{align}
\end{defi}

The symbols we are interested in quantizing to build $D_h (\beta)$ are defined as follows.

\begin{defi}[$S^k(T^*(\C / \Gamma); \C^{2\times 2})$]\label{def:1275}
For $k\in \R$, we say a matrix-valued function $A\in C^\infty ( T^*(\C / \Gamma) ; \C^{2 \times 2})$ is in $S^k(T^*(\C / \Gamma); \C^{2\times 2})$ if
\begin{align}
A(z,\zeta) = \mat{A_{11}(z,\zeta) & A_{12}(z,\zeta) \\ A_{21}(z,\zeta) & A_{22}(z,\zeta)} \label{eq:466}
\end{align}
and $A_{ij} \in S^k(T^*(\C/ \Gamma))$ for each $i,j$.
\end{defi}

We define the Weyl quantization of matrix-valued symbols by quantizing element-wise
\begin{align}
\Op_h^w(A) \coloneqq \mat{\Op_h^w(A_{11}) & \Op_h^w(A_{12}) \\ \Op_h^w(A_{21}) & \Op_h^w(A_{22})}
\end{align}
where $\Op_h^w (A_{ij})$ is defined as
\begin{align}
\Op_h^w (A_{ij} )&u(x_1 + i x_2) \coloneqq\\
&\frac{1}{(2\pi h)^2} \int_{\R^2}\int_{\R^2} e^{\frac{i}{h} \ip{(x-y)}{\xi}} \tilde A_{ij} \left (\frac{x_1 + y_1}{2} + i \frac{x_2+y_2}{2}, \xi_1 + i \xi_2 \right ) u (y_1 + i y_2) \dd y \dd \xi
\end{align}
for $x \in \R^2$, $u\in  \mathcal{S}(\C)$, and $\tilde A_{ij}$ is the $\Gamma$-periodization of $A_{ij}$ in the $z$ variable. 
This integral is well defined for $u\in \mathcal{S}(\C)$.
By duality, it extends to a linear operator from $\mathcal{S}'(\C)$ to itself. 
Note that because each $\tilde A_{ij}$ is periodic on $\Gamma$, $\Op_h^w (A_{ij})$ maps the $\Gamma$-periodic elements of $\mathcal{S}'(\C)$ to the $\Gamma$-periodic elements of $\mathcal S'(\C)$. 

{We next want to show that quantizations of elements of $S^k(T^*(\C/\Gamma); \C^{2\times 2})$ are bounded maps between certain Sobolev spaces on $\C / \Gamma$.
This requires a brief digression into defining such Sobolev spaces.
}

Because $\Gamma \simeq \ZZ^2$, any element $u\in L^2 (\C / \Gamma)$ has a Fourier series representation
\begin{align}
u(z) = \sum _{k \in \Gamma^*} c_k e_k (z) \label{eq:566}
\end{align}
where for each $k\in \Gamma^*$
\begin{align}
e_k (z) \coloneqq (\Vol (\C / \Gamma ) )^{-1/2 }e^{\frac{i}{2} (z\bar k  + \bar z k)}.
\end{align}
These $e_k$ form an orthonormal basis so that
\begin{align}
\norm{u}_{L^2 (\C/ \Gamma )}^2 = \sum_{k\in \Gamma ^* } |c_k|^2. 
\end{align}

For $s\in \R$, define the $h$-dependent Sobolev space $H^s_h(\C /\Gamma)$ as the vector space of elements $u\in L^2 (\C / \Gamma)$ such that
\begin{align}
\norm{ \pi_{\C /\Gamma } \circ \Op_h^w (\jpb{\xi}^s) \circ \tau_{\C / \Gamma } u}_{L^2 (\C / \Gamma)} < \infty
\end{align}
where $\jpb{\xi} \coloneq (1+|\xi|^2)^{1/2}$, $\pi_{\C / \Gamma}\colon L^2_{\operatorname{loc}} (\C ) \to L^2 (\C / \Gamma)$ is the restriction to the fundamental domain of $\Gamma$ centered at the origin (which we abuse and call $\C/\Gamma$), $\tau_{\C / \Gamma} \colon L^2 (\C / \Gamma) \to L^2_{\operatorname{loc}}(\C)$ is the $\Gamma$-periodization of an element of $L^2 (\C / \Gamma) $ to $L^2_{\operatorname{loc}}(\C)$, and $\Op_h^w$ denotes the usual Weyl-quantization on $\R^{2}$ which is identified with $\C$.

If $u$ is written in the form \eqref{eq:566} and $s\in \R_{\ge 0}$, then $u\in H^s_h(\C/\Gamma)$ if and only if 
\begin{align}
\sum _{k\in \Gamma ^*} \abs{hk}^{2s} |c_k|^2 < \infty. \label{eq:577}
\end{align}
We define the $H^s _h(\C / \Gamma)$ norm as the square root of the left-hand side of \eqref{eq:577}.

\begin{prop}\label{prop:1325}
Suppose $k,s\in \R$ and $A\in S^k (T^* (\C/ \Gamma) ; \C^{2\times 2})$. 
Then $\pi_{\C /\Gamma} \circ \Op_h^w (A) \circ \tau_{\C /\Gamma}$ is a bounded map from $H^s_h (\C/ \Gamma; \C^2)$ to $H^{s-k}_h (\C/ \Gamma;\C^2)$.
\end{prop}

\begin{proof}
Here we adapt the argument from \cite[Theorem 5.5]{zworski2012}.

\noindent \textbf{Step 1. Conjugate the symbol to reduce to $k = s = 0$ }

It suffices to work component-wise. Let $a\in S^k (T^* (\C / \Gamma))$. It suffices to show that 
\begin{align}
\pi_{\C /\Gamma}\Op_h^w (\jpb{\xi}^k) \Op_h^w (a) \Op_h^w  ( \jpb{\xi}^{-s}) \tau_{\C /\Gamma} \colon L^2 (\C / \Gamma ) \to L^2 (\C / \Gamma)
\label{main estimate}
\end{align}
is bounded. By the composition of Weyl-quantizations (see for instance \cite[Theorem 4.18]{zworski2012}), there exists $\tilde b \in S^0(\C) $ such that 
\begin{align}
\Op_h^w (\tilde  b) = \Op_h^w (\jpb{\xi}^k) \Op_h^w (\tau_{\C / \Gamma} a) \Op_h^w  (\jpb{\xi}^{-s}) .
\end{align}
{Because $\tau_{\C / \Gamma} a$ is $\Gamma$-periodic in $\C$ and is conjugated by symbols which are also $\Gamma$-periodic in $\C$, $\tilde b$ is $\Gamma$-periodic in $\C$.
Let $b = \pi_{\C / \Gamma} \tilde b \in S^0(T^* (\C / \Gamma))$.    
}

\noindent \textbf{Step 2. Decompose $\Op_h^w (\tilde b)$}

Let $u\in L^2 (\C /\Gamma)$ and define $\tilde u \coloneqq \tau_{\C / \Gamma}u \in L^2_{\operatorname{loc}}(\C)$. 
We now have, for $x\in \C / \Gamma$, that $\Op_h^w (\tilde b) \tilde u (x)$ equals
\begin{align}
 & \frac{1}{(2\pi h )^2} \int_{\R^2} \int_{\R^2} e^{\frac{i}{h} \ip{(x-y) }{\xi}} \tilde b \left (\frac{x_1 + y_1}{2} + i \frac{x_2+y_2}{2}, \xi_1 + i \xi_2 \right ) \tilde u (y_1 + i y_2) \dd y \dd \xi\\
&= \sum_{k\in \Gamma} B_k \tilde u(x) 
\end{align}
where $B_k \tilde u(x)$ is defined as
\begin{align}
  \frac{1}{(2\pi h )^2} \int_{\R^2} \int_{\C / \Gamma} e^{\frac{i}{h} \ip{(x-y+k) }{\xi}}  b \left (\frac{x_1 + y_1-k_1}{2} + i \frac{x_2+y_2-k_2}{2}, \xi_1 + i \xi_2 \right ) \tilde u (y_1 + i y_2) \dd y \dd \xi
\end{align}
{where $k = k_1 + i k_2$.}

By periodicity of $\tilde b$, we see for each $k\in \Gamma$:
\begin{align}
B_k  = 1 _{\C / \Gamma} T_{-k} \Op_h^w(\tilde b)1 _{\C / \Gamma} \label{eq:1364}
\end{align}
where $T_{-k} v (x): = v(x +k)$ {and $1_{\C / \Gamma}$ is the characteristic function on the fundamental domain of $\Gamma$. }

\noindent \textbf{Step 3. Bound each component $B_k$ for $k$ away from zero}

For each $N\in \N$, {$x$ and $y$ in the fundamental domain of $\Gamma$, and $|k| > 8\pi$ (so that $|x -y  +k| \neq 0$)}, we can write
\begin{align}
e^{\frac{i}{h} \ip{x-y+k}{\xi}} = h^{2N} |x- y+k |^{-2N} |D_\xi|^{2N} e^{\frac{i}{h} \ip{x-y+k}{\xi}}
\end{align}
so that by integration by parts (using that the Fourier transform of compactly supported functions will be Schwartz in $\xi$)
\begin{align}
B_k = 1_{\C /\Gamma} T_{-k} \tilde B_k 1 _{\C / \Gamma}
\end{align}
with $\tilde B_k u(x)$ defined as
\begin{align}
\frac{1}{(2\pi h)^2} \int_{\R^2} \int_{\R^2} e^{\frac{i}{h} \ip{x-y}{\xi}} \chi(x-k) \chi(y ) h^{2N} |x- y|^{-2N} (|D_\xi|^{2N} \tilde b) \left (\frac{x+y}{2}, \xi \right ) \tilde u(y) \dd y \dd \xi
\end{align}
with $\chi  \in C_0^\infty(\R^2)$ identically $1$ near the fundamental domain of $\Gamma$ (identified with $\R^2$).
{We can then apply the Schur test to get that there exists a $C>0$ independent of $k$ such that
\begin{align}
    \norm{\tilde B_k}_{L^2 (\R^2)\to L^2 (\R^2)} \le C h^{2N}|k|^{-2N}
\end{align}
so that $\norm{B_k}_{L^2 (\C / \Gamma) \to L^2 (\C / \Gamma)}= \mathcal{O}(h^{\infty}|k|^{-\infty})$.}

{From \eqref{eq:1364}, $B_k$ is bounded from $L^2(\C / \Gamma)$ to $L^2 (\C / \Gamma)$ for $|k| \le 8\pi$.
This, combined with the $L^2$ estimate for $B_k$ away from zero gives  \eqref{main estimate}.
}\end{proof}

We therefore have the following quantization.

\begin{defi}[Quantization of $S^k(T^*(\C / \Gamma); \C^{2\times 2})$]\label{def:quantization_thing}
For $k\in \R$ and $$A \in S^k(T^*(\C / \Gamma); \C^{2\times 2}),$$ we define 
\begin{align}
    \Op_h(A) \coloneqq \pi_{\C / \Gamma }\circ \Op_h^w (A) \circ \tau_{\C / \Gamma} 
\end{align} 
\end{defi}
{We note that for $A \in S^k(T^*(\C / \Gamma); \C^{2\times 2})$, we get by Proposition \ref{prop:1325} that $$\Op_h(A) \colon H^k (\C /\Gamma ; \C^2 )\to L^2 (\C / \Gamma ; \C^2).$$}

For $k \in \R$, matrix-valued symbols $A, A_j \in S^k (T^*(\C/\Gamma); \C^{2\times 2})$ ($j\in \Z_{\ge0}$) depending on $h$, we write
\begin{align}
A \sim \sum _0^\infty h^j A_j \label{eq:1231}
\end{align}
in $S^k (T^* (\C / \Gamma );\C^{2\times 2}))$ if for all $J \in \N$
\begin{align}
A - \sum_{j=0}^J h^j A_j \in  h^{J+1} S^k (T^*(\C / \Gamma) ; \C^{2\times 2}).
\end{align}

\begin{defi}[Principal symbol]\label{def:principal symbol}
    If $A \in S^k(T^*(\C / \Gamma ) ; \C^{2\times 2})$ has an asymptotic expansion as in \eqref{eq:1231}, then its \textit{principal symbol} is the equivalence class of matrix-valued functions $A_0$ with the equivalence relation: \begin{align}
        A\sim B \iff A - B \in h S^k(T^* (\C / \Gamma) ; \C^{2\times 2}).
    \end{align}
\end{defi}

\begin{prop}[Composition]\label{claim:composition}
Suppose $k,\ell \in \R$,
\begin{align}
 A \in S^k (T^* (\C /\Gamma) ; \C^{2\times 2}) && \text{and}&& B \in S^\ell (T^* (\C /\Gamma) ; \C^{2\times 2}).
\end{align}
Then there exists a $C \in S^{k + \ell} (T^* (\C /\Gamma) ; \C^{2\times 2})$ such that:
\begin{align}
\Op_h (A) \circ \Op_h (B) = \Op_h (C)
\end{align}
and
\begin{align}
C \sim AB + \sum _{1}^\infty h^j   \mat{(C_{11})_j & (C_{12})_j \\ (C_{21})_j & (C_{22})_j}.
\end{align}
\end{prop}
\begin{proof}
By standard results about the composition of Weyl operators (see for instance \cite[\S 4.3]{zworski2012}), we have that
\begin{align}
\Op_h^w (\tilde A) \circ \Op_h^w (\tilde B) = \Op _h ^w (\tilde C)
\end{align}
where $\tilde C  = \tilde A  \# \tilde B = \tilde A \tilde B + \mathcal{O}(h) \in S^{k + \ell } (T^* (\C ) ; \C^{2\times 2})$, {and $\tilde A$ and $\tilde B$ denotes the periodization in the $z$ variable of all matrix components.
It is easy to verify from the formula for $a\#b$ (see, for instance, \cite[Theorem 4.12]{zworski2012}) that if $a$ and $b$ are $\Gamma$ periodic in $z$, then $a\#b$ is $\Gamma$ periodic in $z$.
Therefore, we can restrict $\tilde C$ to the fundamental domain of $\Gamma$ in the variable $z$, which we denote by $C \in S^{k+ \ell } (T^*(\C / \Gamma) ; \C^{2\times 2})$.
}

If $u\in H^{k+ \ell} (\C / \Gamma ; \C^2)$ and $\tilde u \coloneqq \pi_{\C / \Gamma}^{-1} u$, then:
\begin{align}
\Op_h(A) \circ \Op_h(B) u &= (\Op_h^w (A)\Op_h^w(B) \tilde u  )|_{\C /\Gamma}\\
&= (\Op_h ^w(C) \tilde u )|_{\C / \Gamma} = \Op_h(C) u. 
\end{align}
\end{proof}

\begin{prop}[Parametrix construction]\label{claim.parametrix}
If $k\in \R$ and $F \in S^k (T^*(\C / \Gamma) ; \C^{2\times 2})$ with principal symbol
\begin{align}
\mat{A(z,\zeta)  & B(z,\zeta) \\ C(z,\zeta) & D(z,\zeta)}
\end{align}
such there exists a $c_0> 0 $ with $A(z,\zeta)D(z,\zeta) - B(z,\zeta)C(z,\zeta) > c_0$.
Then there exists $G \in S^{-k} (T^*(\C / \Gamma) ; \C^{2\times 2})$ such that
\begin{align}
\Op_h (G) \circ \Op_h(F) = \Op_h (F) \circ \Op_h (G) = O(h^{\infty})
\end{align}
and the principal symbol of $G$ is
\begin{align}
\mat{A(z,\zeta)  & B(z,\zeta) \\ C(z,\zeta) & D(z,\zeta)}^{-1}.
\end{align}
\end{prop}
\begin{proof}
Because we have a composition rule for this symbol class (Proposition \ref{claim:composition}), the claim follows by an identical argument as the usual parametrix construction, see, for instance, \cite[Theorem 4.1]{Grigis}.
\end{proof}

\printbibliography

\end{document}